\documentclass{siamart190516}

\usepackage[utf8]{inputenc}

\usepackage{amssymb}
\usepackage{bm}
\usepackage{mathtools}
\usepackage{booktabs}
\usepackage{tikz}
\usepackage{hyperref}
\usepackage{comment}
\usepackage{enumerate}
\usepackage{etoolbox}
\usepackage{siunitx}
\usepackage{multirow}
\usepackage{paralist}
\usepackage{microtype}
\usepackage{textcomp}

\usetikzlibrary{angles, quotes}

\tikzset{%
  >=latex,
  disk/.style={draw, circle, inner sep=0},
  maskedbg/.style={rectangle, fill=white, opacity=0.8, text opacity=1, inner sep=0}
}

\newsiamremark{remark}{Remark}

\headers{Approximation of Local Expansions by the FMM}{Matt Wala and Andreas Klöckner}

\def\today{August 2, 2020}

\title{On the Approximation of Local Expansions of Laplace Potentials by the Fast Multipole Method\thanks{Submitted to the editors on~\today.}}

\author{Matt Wala%
  \thanks{Department of Computer Science, University of Illinois at Urbana-Champaign,
    201 N. Goodwin Ave, Urbana, IL 61801
    (\email{wala1@illinois.edu}, \email{andreask@illinois.edu}).}
  \and
  Andreas Klöckner\footnotemark[2]
}

\date{\today}

\newcommand{\norm}[1]{\left\lvert #1 \right\rvert}

\newcommand{\bignorm}[1]{\left\lVert#1\right\rVert}

\newcommand{\pot}[1]{\phi_{#1}}
\newcommand{\mpole}[2]{\mathcal{M}_{#1}^{#2}}
\newcommand{\local}[2]{\mathcal{L}_{#1}^{#2}}
\newcommand{\ilocloc}[4]{\substack{{#1} \in I(0, {#2}) \\ \cap I({#3}, {#4})}}
\newcommand{\tsup}[1]{\ensuremath{\sp{\text{#1}}}}
\newcommand{\transpose}[1]{#1\tsup{T}}
\newcommand{\htranspose}[1]{#1\tsup{H}}

\definecolor{qbxcolor}{RGB}{43,131,186}
\definecolor{srccolor}{RGB}{63,140,52}
\definecolor{localcolor}{RGB}{215,25,28}

\newcommand{\ilist}[2]{%
  \ifstrequal{#1}{1}{U_{#2}}{%
  \ifstrequal{#1}{2}{V_{#2}}{%
  \ifstrequal{#1}{3}{W_{#2}}{%
  \ifstrequal{#1}{3close}{W^\mathrm{close}_{#2}}{%
  \ifstrequal{#1}{3far}{W^\mathrm{far}_{#2}}{%
  \ifstrequal{#1}{4}{X_{#2}}{%
  \ifstrequal{#1}{4close}{X^\mathrm{close}_{#2}}{%
  \ifstrequal{#1}{4far}{X^\mathrm{far}_{#2}}{}%
}}%
}}%
}}%
}}

\newcommand{\ptpot}{\phi}

\newcommand{\innerprod}[2]{{(#1, #2)}_{\mathbb{S}^2}}

\newcommand{\quadr}{\mathcal{Q}}

\newcommand{\fourier}[1]{\mathcal{F}_{#1}}
\newcommand{\proj}[1]{\mathcal{P}_{#1}}

\definecolor{qbxcolor}{RGB}{43,131,186}
\definecolor{srccolor}{RGB}{63,140,52}
\definecolor{localcolor}{RGB}{215,25,28}

\begin{document}

\maketitle

\begin{abstract}
In this paper, we present a generalization of the classical error bounds of Greengard--Rokhlin for the Fast Multipole Method (FMM) for Laplace potentials in three dimensions, extended to the case of local expansion (instead of point) targets. We also present a complementary, less sharp error bound proven via approximation theory whose applicability is not restricted to Laplace potentials.
Our study is motivated by the GIGAQBX~FMM, an algorithm for the fast, high-order accurate evaluation of layer potentials near and on the source layer. GIGAQBX is based on the FMM, but unlike a conventional FMM, which is designed to evaluate potentials at point-shaped targets, GIGAQBX evaluates \emph{local expansions} of potentials at ball-shaped targets. Although the accuracy (or the \emph{acceleration error}, i.e., error due to the approximation of the potential by the fast algorithm) of the conventional FMM is well understood, the acceleration error of FMM-based algorithms applied to the  evaluation of local expansions has not been as well studied.
The main contribution of this paper is a proof of a set of hypotheses first demonstrated numerically in the paper \emph{A Fast Algorithm for Quadrature by Expansion in Three Dimensions}, which pertain to the accuracy of FMM approximation of local expansions of Laplace potentials in three dimensions. These hypotheses are also essential to the three-dimensional error bound for GIGAQBX, which was previously stated conditionally on their truth and can now be stated unconditionally.
\end{abstract}

\begin{keywords}
  expansion, fast multipole method, integral equations,
  spherical harmonics, Laplace equation
\end{keywords}

\begin{AMS}
  65R20, 
  65T99, 
  42C10, 
  31B10, 
  65D32  
\end{AMS}

\section{Introduction}

A variety of numerical methods for problems in computational physics frequently require the evaluation of potentials due to a set of $N$ sources at a set of $M$ targets, where $N$ and $M$ may be large. The pervasive concern in the implementation of such methods is reduction of the otherwise $O(NM)$ computational cost involved in this operation. Fast algorithms such as the Fast Multipole Method ~(FMM,~\cite{carrier:1988:adaptive-fmm}) are used to accomplish this, enabling methods with computational scaling that is linear in the number of sources and targets. Most FMMs evaluate \emph{point potentials}, which are sums of the form
\begin{equation}
    \label{eqn:pt-pot}
    \Phi(\bm{t}_j)
    =
    \sum_{i=1}^N w_i \mathcal{G}(\bm{t}_j, \bm{s}_i), \quad j = 1, \ldots, M.
\end{equation}
Here, $\{\bm{s}_i\} \subset \mathbb{R}^3$ is a set of source particles, $\{\bm{t}_j\} \subset \mathbb{R}^3$ is a set of target particles, and $\{w_i\} \subset \mathbb{R}$ is a set of source weights. The kernel function $\mathcal{G}(\cdot, \cdot)$ used in this paper is the free-space three-dimensional Laplace Green's function $\mathcal{G}(\bm{x}, \bm{y}) = (4\pi)^{-1} \norm{\bm{x}-\bm{y}}^{-1}$  (or \emph{Laplace kernel}). The approximation to~\eqref{eqn:pt-pot} produced by the FMM comes with a number of guarantees, including $O(N+M)$ evaluation complexity and mathematically rigorous error bounds. Recent work in fast algorithms for integral equation methods~\cite{rachh:2017:qbx-fmm,wala:2018:gigaqbx2d,wala:2019:gigaqbx3d} uses a modified FMM to evaluate a related but distinct potential, namely the \emph{local expansion} of~\eqref{eqn:pt-pot}. The local expansion is a modified potential that arises after separation of variables via the addition theorem for the Laplace kernel (cf.~\eqref{eqn:laplace-addition-thm-solid}) and takes the form
\begin{equation}
  \label{eqn:locexp}
  \local{\bm{c}_j}{p}[\Phi](\bm{t}_j)
  =
  \sum_{n=0}^p \sum_{m=-n}^n
  L_n^m \norm{\bm{t}_j - \bm{c}_j}^n
  Y_n^m(\widehat{\bm{t}_j - \bm{c}_j}),
  \quad
  j = 1, \ldots, M,
\end{equation}
where $\{\bm{c}_j\} \subset \mathbb{R}^3$ is a set of expansion centers, $\{Y_n^m\}$ are spherical harmonics, the notation $\bm{\hat{v}}$ refers to the normalized vector $\bm{v}$, and the (center-dependent) \emph{local coefficients} $\{L_n^m\}$ are defined as
\[
  L_n^m = \frac{1}{2n+1} \sum_{i=1}^N w_i
  \frac{\overline{Y_n^m(\widehat{\bm{s}_i - \bm{c}_j})}}{\norm{\bm{s}_j - \bm{c}_j}^{n+1}}.
\]
Standard FMM error estimates do not immediately apply to the approximation of~\eqref{eqn:locexp}. The aim of the present paper is to give bounds for the approximation of~\eqref{eqn:locexp} via the FMM, by generalizing the estimates of Greengard--Rokhlin for three-dimensional point potentials~\cite{greengard:1988:thesis} to the case of local expansions. Specifically, we provide proofs for a set of hypotheses first demonstrated numerically in~\cite{wala:2019:gigaqbx3d}.

One context where the evaluation of local expansions such as~\eqref{eqn:locexp} arises is in integral equation methods for the solution of boundary value problems of elliptic constant-coefficient partial differential equations (PDEs). A key representation of solutions to PDEs in integral equation methods is a layer potential, such as the single-layer potential,
an integral operator
\begin{equation}
  \label{eqn:slp}
  \mathcal{S}\sigma(\bm{x}) = \int_{\Gamma} \mathcal{G}(\bm{x}, \bm{y}) \sigma(\bm{y}) \, dS(\bm{y}), \quad \bm{x} \in \mathbb{R}^3,
\end{equation}
defined over a bounded surface $\Gamma \subseteq \mathbb{R}^3$, with a density function $\sigma: \Gamma \to \mathbb{C}$. Effective numerical realization of integral equation methods requires addressing two interlocking concerns: \emph{accurate quadrature} for on-surface evaluation of layer potentials and \emph{acceleration} to reduce the cost of evaluating discretized layer potentials. When discretized (e.g.) using a panel-based `smooth' composite quadrature rule, the quadrature approximation from such a discretization is known to be inaccurate when $\bm{x}$ is close to or on the surface, due to the singularity of the Laplace kernel on its diagonal. Therefore, discretization of operators such as $\mathcal{S}$ requires a corrected quadrature that can be applied rapidly at a potentially large number of on-surface targets.

A general technique for accurate on-surface/near-surface layer potential evaluation is Quadrature by Expansion (QBX,~\cite{klockner:2013:qbx}). QBX is based on the observation that the potential admits a smooth high-order approximation via local expansions that can be treated via standard quadrature. It is applied as follows. Consider a generic quadrature discretization of the single-layer potential
\begin{equation}
  \label{eqn:slp-quad}
  \quadr [\mathcal{S} \sigma] (\bm{x})
  = \sum_{i=1}^N w_i \mathcal{G}(\bm{x}, \bm{y}_i) \sigma(\bm{y}_i)
\end{equation}
where the weights $\{w_i\}$ arise from a smooth quadrature rule for surfaces. If the target $\bm{x}$ is on the surface or is near $\Gamma$, then the above approximation may not be accurate due to the near-singularity of the integrand. To correct this, an off-surface expansion center $\bm{c} \in \mathbb{R}^3 \setminus \Gamma$ close to $\bm{x}$ is chosen. A truncated local expansion~\eqref{eqn:locexp} is then formed which stands in for the original potential.
Intuitively, this local expansion is a smoother function than the original potential at the target, and hence is easier for quadrature to handle. In combination with a suitable smooth quadrature rule, QBX recovers a high-order approximation to $\mathcal{S} \sigma$. For more details, see~\cite{klockner:2013:qbx, epstein:2013:qbx-error-est}.

The primary research direction in the acceleration of QBX has been to find appropriate modifications to the FMM to permit for the evaluation of QBX local expansions, despite the fact that the local expansion~\eqref{eqn:locexp} is not quite a point potential~\eqref{eqn:pt-pot}.  Part of what makes the task nontrivial, as previously mentioned, is that standard FMM error estimates do not apply when the local expansion is being approximated by the FMM\@. Thus, a major concern in accelerated QBX schemes is that QBX expansions are formed accurately.

The GIGAQBX~FMM~\cite{wala:2018:gigaqbx2d,wala:2019:gigaqbx3d,wala:2020:gigaqbxts} is a recently developed technique for FMM-accelerated evaluation of QBX expansions. Compared with other FMMs, GIGAQBX differs by treating local expansions essentially as ball-shaped `targets with extent,' suspending them at non-leaf levels of the tree where they protrude beyond their bounding box by a defined amount, and modifying the appropriate FMM interaction when necessary to guarantee the accuracy of FMM evaluation.

The subject of this paper is closely tied to the error analysis of GIGAQBX\@.
We refer to the error studied in this paper as \emph{acceleration error} to distinguish the error in the fast algorithm from other sources of error in layer potential evaluation. In a conventional point FMM, acceleration error is the difference between the point potential and the approximation formed by the FMM\@. In GIGAQBX, acceleration error is the difference between the local expansion of a point potential~\eqref{eqn:locexp} and the local expansion of its approximation formed by the fast algorithm, as measured by point evaluations of the expansions within a certain distance from the center. Straightforwardly, this approximate potential equals the local expansion of an approximate point potential. In other words, to model the acceleration error $e_{\textrm{accel}}$ of GIGAQBX in evaluating the potential $\Phi$, one can regard the output of GIGAQBX as an exact local expansion of an approximate point potential $\mathcal{A}^p_{\textrm{FMM}}[\Phi]$, where $p$ is an intermediate expansion order:
\begin{equation}
    \label{eqn:accel-error}
    e_{\textrm{accel}}=
    \left|
    \local{\bm{c}}{q}[\Phi](\bm{x})
    -
    \local{\bm{c}}{q}[\mathcal{A}^p_{\textrm{FMM}}[\Phi]](\bm{x})
    \right|,
\end{equation}
and where $\bm{c}$ is the expansion center associated to a target $\bm{x}$.

In the two-dimensional case, analytical acceleration error bounds for GIGAQBX were presented in~\cite{wala:2018:gigaqbx2d}. For a unit strength point source, these bounds imply an error of the form $O(q(1/2)^{p+1})$ where $p$ is the intermediate expansion order and $q$ is the QBX order. The techniques in that paper do not extend to three dimensions. The paper~\cite{wala:2019:gigaqbx3d} introduced the three-dimensional GIGAQBX FMM, producing the above-mentioned series of numerical hypotheses which imply an error bound of the form $O((3/4)^{p+1})$. For a rough comparison, when using a near-neighborhood size of one box width, the point version of the FMM achieves acceleration error of $O((1/2)^{p+1})$ and $O((3/4)^{p+1})$ in two and three dimensions, respectively~\cite{petersen:1995:fmm-error-est,petersen:1995:fmm-error-est-3d}.

Related work on accelerated QBX has handled acceleration error as follows. The QBX FMM by Rachh et al.~\cite{rachh:2017:qbx-fmm,rachh2015integral} is the first published QBX--FMM coupling and the first to note the nontrivial nature of accurately forming QBX local expansions with the FMM\@. This scheme controls acceleration error by increasing the intermediate FMM order until error tolerance is achieved, an empirically effective procedure not backed by mathematical bounds. The QBKIX scheme~\cite{rahimian:2018:qbkix} does not exhibit the same kind of acceleration error as the QBX~FMM of~\cite{rachh:2017:qbx-fmm} or GIGAQBX, as it uses an entirely different expansion formation mechanism based on the kernel-independent FMM~\cite{ying:kifmm:2004}. However, this technique incurs errors due to aliasing of the expansion coefficients and extrapolation of the resulting expansions. In the local QBX scheme of Siegel and Tornberg~\cite{siegel:2018}, the QBX-mediated near-field is evaluated directly, and the only acceleration error is due to an FMM-mediated point potential.

In summary of the discussion above:
\begin{itemize}
    \item In this paper we prove Hypotheses 1--3 of~\cite{wala:2019:gigaqbx3d}, which enable a high-order error estimate for GIGAQBX in three dimensions.
    \item We show how the error estimates for translation operators of Greengard--Rokhlin are a special case of the above hypotheses.
\end{itemize}
Additionally, we make two further contributions in this paper:
\begin{itemize}
    \item We present an alternative bound for the FMM approximation of local expansions that is based on approximation theory. This bound, though not sharp, holds true for a wide variety of translation operators and kernels.
    \item We validate the usefulness of these bounds by comparison with prior numerical evidence.
\end{itemize}


\section{Background}%
\label{sec:background}

\subsection{Notation}

For a vector $\bm{x} \in \mathbb{R}^3 \setminus \{\bm{0}\}$,
the notation $\bm{\hat{x}}$ refers to the unit
vector in the direction of $\bm{x}$, i.e.,
$
  \bm{\hat{x}} ={\bm{x}}/{\norm{\bm{x}}},
$
where $\norm{\cdot}$ denotes the Euclidean norm.
For a point source $\bm{s} \in \mathbb{R}^3$, the unit-strength potential due to $\bm{s}$ is denoted
$
  \phi_{\bm{s}}(\bm{x}) = {1}/{\norm{\bm{x} - \bm{s}}}.
$
The open Euclidean ball of radius $r \geq 0$ centered at $\bm{c} \in \mathbb{R}^3$ is the set $B(\bm{c}, r) = \{ \bm{x} \in \mathbb{R}^3 : \norm{\bm{x} - \bm{c}} < r \}$.
The closed ball of radius $r$ for the same center is denoted $\bar{B}(\bm{c}, r) = \{ \bm{x} \in \mathbb{R}^3 : \norm{\bm{x} - \bm{c}} \leq r \}$.
The integer interval of radius $r \in \mathbb{N}_0$ centered at $c \in \mathbb{Z} $ is the set $I(c, r) = \{ c-r, c-r+1, \ldots, c+r\}$.
The unit sphere $\{\bm{x} \in \mathbb{R}^3 : \norm{\bm{x}} = 1\}$ is denoted~$\mathbb{S}^2$.

\subsection{Spherical harmonics and Fourier--Laplace series}

A polynomial~$p: \mathbb{R}^3 \to \mathbb{C}$ of degree $n$ is \emph{homogeneous} if it satisfies $p(\lambda \bm{x}) = \lambda^n p(\bm{x})$ for
all $\bm{x} \in \mathbb{R}^3$ and $\lambda \in \mathbb{C}$. The (complex) vector space $\mathbb{Y}_n^3$ (\emph{spherical harmonic space of degree $n$}, cf.~\cite[Def.~2.7]{atkinson:2012:sph-harm})
is the space consisting of the harmonic,
homogeneous polynomials of degree $n$, with domains restricted to the unit sphere.
The \emph{spherical harmonics of degree $n$}, $\{ Y_{n}^{m}: \mathbb{S}^2 \to \mathbb{C} \mid m \in \mathbb{Z}, -n \leq m \leq n \}$, are a basis for $\mathbb{Y}_n^3$, and they are defined (based on~\cite[eq.~(14.30.1)]{nist:dlmf}) as follows, where $0 \leq \theta \leq \pi$ and $0 \leq \phi \leq 2\pi$ are angles:
\[
Y_{n}^m(\bm{\xi}) = \sqrt{\frac{2n+1}{4\pi} \frac{(n-m)!}{(n+m)!}}
e^{i m \phi} P_n^m(\cos(\theta)), \quad \bm{\xi} = \transpose{(\sin \theta \cos \phi, \sin
\theta \sin \phi, \cos \theta)}.
\]
The function $P_n^m$ above is the associated Legendre function~\cite[eq.~(14.3.1)]{nist:dlmf} of degree $n$ and order $m$. (This choice of orthonormal basis for spherical harmonics differs from others in the FMM literature, e.g., \cite{greengard:1988:thesis,wala:2019:gigaqbx3d,siegel:2018}. However, our main results are independent of basis.) The set of \emph{spherical harmonics} is the set $\bigcup_{n=0}^\infty \{Y_n^m \mid -n \leq m \leq n\}$.

A key property of the spherical harmonics is that they form an
orthonormal basis of $L^2(\mathbb{S}^2)$, with the inner product given by $\innerprod{f}{g} = \int_{\mathbb{S}^2} f \overline{g} \, dS$.
For $f \in L^2(\mathbb{S}^2)$,
let $\proj{n}[f]$ be the orthogonal projection of $f$ onto $\mathbb{Y}_n^3$.
The \emph{Fourier--Laplace series} is an orthonormal expansion defined as
\[
\fourier{p}[f](\bm{\xi}) = \sum_{n=0}^p \proj{n}[f](\bm{\xi})=
\sum_{n=0}^p \sum_{m=-n}^n \innerprod{f}{Y_n^m} Y_n^m(\bm{\xi}),
\quad \bm{\xi} \in \mathbb{S}^2.
\]
While it is possible to extend the definition of a Fourier--Laplace series to functions defined on spheres of any radius, in this paper when we speak of Fourier--Laplace series it shall be exclusively for functions defined on the unit sphere.

We will refer to the formulas below to simplify computations with the spherical harmonics. First, the pointwise values of the spherical harmonics along the
positive $z$-axis are particularly simple. These are given by (cf.~\cite[eq.~(14.30.4)]{nist:dlmf})
\begin{equation}
\label{eqn:sph-harm-z-axis}
Y_n^m(\transpose{(0, 0, 1)}) = \begin{dcases}
\sqrt{\frac{2n+1}{4\pi}}, & m = 0, \\
0, & m \neq 0.
\end{dcases}
\end{equation}
Second, a key identity for spherical harmonics is the addition theorem,
which states~\cite[eq.~(14.30.9)]{nist:dlmf}
\begin{equation}
  \label{eqn:sph-harm-addition-thm}
  P_n(\bm{\xi} \cdot \bm{\eta})
  = \frac{4\pi}{2n+1} \sum_{m=-n}^{n} Y_n^m(\bm{\xi})
  \overline{Y_n^m(\bm{\eta})},
  \quad \bm{\xi}, \bm{\eta} \in \mathbb{S}^2.
\end{equation}
The function $P_n$ is the Legendre polynomial of degree $n$.
In particular, this last identity implies
\begin{equation}
  \label{eqn:sph-harm-addition-corollary}
  \frac{2n+1}{4\pi} = \sum_{m=-n}^{n} \norm{Y_n^m(\bm{\xi})}^2,
  \quad \bm{\xi} \in \mathbb{S}^2.
\end{equation}
See~\cite{atkinson:2012:sph-harm,axler:harmonic-functions:2001,dai:2013:approximation} for further details concerning spherical harmonics.

\subsection{Solid harmonics}

The functions $R_n^m: \mathbb{R}^3 \to \mathbb{C}$ and $I_n^m:
\mathbb{R}^3 \setminus \{\bm{0}\} \to \mathbb{C}$ defined by
\begin{align*}
  R_n^m(\bm{x}) &= \norm{\bm{x}}^n Y_n^m(\bm{\hat{x}}), \\
  I_n^m(\bm{x}) &= \norm{\bm{x}}^{-(n+1)} Y_n^m(\bm{\hat{x}}), \quad \bm{x} \neq \bm{0},
\end{align*}
are respectively called the \emph{regular} and \emph{irregular} \emph{solid harmonics}, or \emph{solid harmonics} for short. (Our definition normalizes them such that $\int_{\mathbb{S}^2} \norm{R_n^m}^2 dS = \int_{\mathbb{S}^2} \norm{I_n^m}^2 dS = 1$.) The definition of $R_n^m$ is to be understood so that $R_n^m(\bm{0}) = 0$ for $n > 0$ and $R_0^0(\bm{x}) = 1/\sqrt{4\pi}$.
These functions are
solutions to Laplace's equation and play a key role in expanding Laplace potentials. In particular, the addition theorem for the Laplace kernel states~\cite[eq.~(5.10)]{beatson:greengard:1997} that, for $\bm{x}, \bm{y} \in
\mathbb{R}^3$, $\norm{\bm{x}} < \norm{\bm{y}}$,
\begin{equation}
  \label{eqn:laplace-addition-thm}
   \frac{1}{\norm{\bm{x} - \bm{y}}}
   =
   \sum_{n=0}^\infty \norm{\bm{x}}^n
   \norm{\bm{y}}^{-(n+1)} P_{n}(\bm{\hat{x}} \cdot \bm{\hat{y}}).
\end{equation}
Using~\eqref{eqn:sph-harm-addition-thm}, we write this as
\begin{equation}
  \label{eqn:laplace-addition-thm-solid}
   \frac{1}{\norm{\bm{x} - \bm{y}}}
   =
   \sum_{n=0}^\infty
   \frac{4\pi}{2n+1}
   \sum_{m=-n}^n
   R_n^m(\bm{x}) \overline{I_n^{m}(\bm{y})}.
\end{equation}

We next give addition theorems for expressing shifted
solid harmonics $R_n^m(\bm{x} + \bm{y})$ and $I_n^m(\bm{x} +
\bm{y})$ in a series of unshifted solid harmonic functions of $\bm{x}$. Define a normalization constant
\[
A_n^m = \sqrt{\frac{2n+1}{4\pi} \frac{(n-m)!}{(n+m)!}}.
\]
Then for all $\bm{x}, \bm{y} \in \mathbb{R}^3$
(cf.~\cite[eq.~(7)]{caola:1978:solid-harmonics})
\begin{equation}
\label{eqn:rnm-addition-thm}
R_n^m(\bm{x} + \bm{y}) =
A_n^m \sum_{\nu=0}^n
\sum_{\ilocloc{\mu}{\nu}{m}{n-\nu}}
\binom{n+m}{\nu+\mu}
\frac{R_{n-\nu}^{m-\mu}(\bm{y})
R_{\nu}^{\mu}(\bm{x})} { A_{n-\nu}^{m-\mu} A_{\nu}^{\mu}}.
\end{equation}
Similarly, for $\bm{x}, \bm{y} \in \mathbb{R}^3$ with $\norm{\bm{x}} < \norm{\bm{y}}$
(cf.~\cite[eq.~(10)]{caola:1978:solid-harmonics})
\begin{equation}
\label{eqn:inm-addition-thm}
I_n^m(\bm{x} + \bm{y}) =
A_n^m \sum_{\nu=0}^\infty
\sum_{\mu \in I(m, \nu)}
(-1)^{\nu-m-\mu}
\binom{n+\nu-\mu}{n-m}
\frac{I^{\mu}_{n+\nu}(\bm{y})
R_{\nu}^{m-\mu}(\bm{x})} {A_{n+\nu}^{\mu} A_{\nu}^{m-\mu}}.
\end{equation}

\subsection{Local and multipole expansions}



Local and multipole expansions that may be familiar from the FMM are a type of series based on the addition theorem for the Laplace kernel or the addition theorems for solid harmonic functions. In this paper, we will sometimes prefer to regard local and multipole expansions as (integral) \emph{operators} $\local{\bm{c}}{p}[\cdot]$ and $\mpole{\bm{c}}{p}[\cdot]$, acting on functions that solve a Dirichlet boundary value problem of the Laplace equation, yielding a series representing other such functions that, after a change of domain, is a Fourier--Laplace series.
The series that result in both cases---via addition theorems or integrals---are equivalent (cf.~Remark~\ref{rem:fl} below), but the chief advantage of defining expansions via integral operators is that it makes certain analytical properties conveniently apparent. In this section, we recall the details of the operator definition.

\subsubsection{Local expansions}

We show how the local expansion operator may be defined starting with the Poisson integral identity for the unit ball. The \emph{Poisson kernel} for the three-dimensional unit ball~\cite[eq.~(2.127)]{atkinson:2012:sph-harm}, is given by
\begin{equation}
  \label{eqn:poisson-kernel}
  \mathcal{P}(r, t) = \frac{1}{4\pi}  \frac{1-r^2}{(1+r^2-2rt)^{3/2}}, \quad r \in (-1, 1), \enskip t \in [-1, 1],
\end{equation}
and it also~\cite[Prop.~2.28]{atkinson:2012:sph-harm} has the Legendre series expansion
\begin{equation}
  \label{eqn:poisson-kernel-legendre} \mathcal{P}(r, t) = \frac{1}{4\pi} \sum_{n=0}^\infty (2n+1) r^n P_n(t), \quad
  r \in (-1, 1), \enskip t \in [-1, 1].
\end{equation}
Let $g: \bar{B}(\bm{0}, 1) \to \mathbb{R}$ be a function that is harmonic inside the unit ball and continuous on the closed unit ball. (The assumption of continuity on the boundary can be relaxed, though it is more than sufficient for purposes of the functions we use in this paper.) The Poisson integral identity~\cite[Thm.~2.2.5 (adjusted for normalization)]{dai:2013:approximation} states:
\[
  g(\bm{x}) = \int_{\mathbb{S}^{2}} \mathcal{P} \left( \norm{\bm{x}}, \bm{\hat{x}} \cdot \bm{\xi} \right) g(\bm{\xi}) \, dS(\bm{\xi}),
  \quad \norm{\bm{x}} < 1.
\]
Suppose that $f$ is harmonic inside a ball in of radius $\rho > 0$ centered at $\bm{c} \in \mathbb{R}^3$ and is continuous on the closure of the same ball. A Poisson integral representation of $f$ may be obtained by considering the function with a translated and scaled domain
\begin{equation}
    \label{eqn:translated-scaled-function}
    \mathcal{T}_{\bm{c},\rho}[f](\bm{y}) = f(\bm{c} + \rho\bm{y}), \quad \norm{\bm{y}} \leq 1.
\end{equation}
Observe the function $\mathcal{T}_{\bm{c},\rho}[f]$ is harmonic inside the unit ball.
For $\bm{x} \in B(\bm{c}, \rho)$,
we may represent $f(\bm{x})$ using
\begin{equation}
  \label{eqn:poisson-formula-generic-ball}
  f(\bm{x}) = \mathcal{T}_{\bm{c},\rho}[f]\left(\frac{\bm{x}-\bm{c}}{\rho}\right) = \int_{\mathbb{S}^{2}}
  \mathcal{P}\left(\frac{\norm{\bm{x}-\bm{c}}}{\rho}, (\widehat{\bm{x}-\bm{c}}) \cdot \bm{\xi} \right)
  f(\bm{c} + \rho \bm{\xi}) \, dS(\bm{\xi}).
\end{equation}
We substitute the Legendre series expansion of the Poisson kernel~\eqref{eqn:poisson-kernel-legendre}
into the above formula and interchange the order of integration and summation (which is permitted due to uniform convergence of the series for $\mathcal{P}(r, \cdot)$ for a fixed radius $r$), so that
\begin{equation}
  \label{eqn:poisson-integral-partial-expansion}
  f(\bm{x}) =
  \sum_{n=0}^\infty
\frac{\norm{\bm{x}-\bm{c}}^n}{\rho^n}
  \left(\frac{2n+1}{4\pi}\right)
  \int_{\mathbb{S}^{2}}
  P_n\left((\widehat{\bm{x} - \bm{c}}) \cdot \bm{\xi}\right)
  f(\bm{c} + \rho \bm{\xi}) \, dS(\bm{\xi}).
\end{equation}
Next, we apply the spherical harmonic addition theorem~\eqref{eqn:sph-harm-addition-thm} to expand the Legendre polynomial terms, obtaining
\[
  f(\bm{x}) = \sum_{n=0}^\infty
  \frac{\norm{\bm{x}-\bm{c}}^n}{\rho^n}
  \int_{\mathbb{S}^{2}}
  \sum_{m=-n}^{n}
  Y_{n}^{m}\left(\widehat{\bm{x}-\bm{c}}\right)
  \overline{Y_{n}^{m}(\bm{\xi})}
  f(\bm{c} + \rho \bm{\xi}) \, dS(\bm{\xi}).
\]
Defining \emph{local coefficients} via
\[
  L_{n}^{m} = \frac{1}{\rho^n} \int_{\mathbb{S}^{2}} f(\bm{c} + \rho \bm{\xi}) \overline{Y_{n}^{m}(\bm{\xi})} \, dS(\bm{\xi}),
\]
we call the \emph{local expansion} of $f$ the series representation
\begin{equation}
  \label{eqn:convergent-local-expansion}
  f(\bm{x}) = \sum_{n=0}^{\infty}
  \sum_{m=-n}^{n}
  L_{n}^{m}
  \norm{\bm{x} - \bm{c}}^n
  Y_{n}^{m}\left(\widehat{\bm{x}-\bm{c}}\right)
  =
  \sum_{n=0}^\infty
  \sum_{m=-n}^{n}
  L_{n}^{m}
  R_n^m (\bm{x} - \bm{c}).
\end{equation}
The next definition summarizes this construction.
\begin{definition}[Local expansion]%
\label{def:local}%
Let $\rho > 0$ and $\bm{c} \in \mathbb{R}^3$.  Let $f :
  \bar{B}(\bm{c}, \rho) \to \mathbb{C}$ be harmonic inside $B(\bm{c}, \rho)$
  and continuous on $\bar{B}(\bm{c}, \rho)$. Define the local coefficients of $f$ via the integrals
  \begin{equation}
  \label{eqn:local-coeff}
     L_{n}^{m} = \frac{1}{\rho^n} \int_{\mathbb{S}^{2}} f(\bm{c} + \bm{\xi} \rho) \overline{Y_n^m(\bm \xi)} dS(\bm \xi).
  \end{equation}
  The function
  \begin{equation}
    \label{eqn:local-exp}
    \local{\bm{c}}{p}[f](\bm{x})
    = \sum_{n=0}^p
    \sum_{m=-n}^{n}
    L_{n}^{m}
    R_n^m(\bm{x} - \bm{c}),
    \quad \bm{x} \in \bar{B}(\bm{c}, \rho),
  \end{equation}
  is called a \emph{$p$-th order local expansion of $f$ centered at $\bm{c}$}.
\end{definition}
\begin{remark}[Connection to Fourier--Laplace series]%
\label{rem:fl}%
The local expansion of a function $f: \bar{B}(\bm{c}, \rho) \to \mathbb{C}$ satisfying the hypotheses of Definition~\ref{def:local} is closely connected to a Fourier--Laplace series:
\begin{equation}
\label{eqn:fourier-laplace-local-connection}
\local{\bm{c}}{p}[f](\bm{x}) = \fourier{p}
\left[\mathcal{T}_{\bm{c},r}
\left[
\left.
f
\right|_{\{\bm{y}: \norm{\bm{y}-\bm{c}} = r\}}
\right]
\right]
(\widehat{\bm{x}-\bm{c}}),
\quad
\norm{\bm{x}-\bm{c}} = r,
\end{equation}
where $\mathcal{T}_{\bm{c},r}[\cdot]$ is defined in~\eqref{eqn:translated-scaled-function}.
In other words, the local expansion of $f$, restricted to a sphere of fixed radius about the expansion center, coincides with the Fourier--Laplace series of $f$ after a scaling and translation of the domain to the unit sphere. The Fourier--Laplace coefficients for the series corresponding to~\eqref{eqn:fourier-laplace-local-connection} are $\{r^n L_n^m\}$.
\end{remark}

It follows from Remark~\ref{rem:fl} that the local coefficients $\{L_{n}^{m}\}$ are the unique coefficients for which the representation~\eqref{eqn:convergent-local-expansion} holds for all $\bm{x} \in B(\bm{c}, \rho)$. For if $\{{(L')}_{n}^{m}\}$ is a second set of local coefficients for $f$ on $B(\bm{c}, \rho)$, then $\{ \norm{\bm{x}-\bm{c}}^n (L_n^m - (L')_n^m)\}$ are Fourier--Laplace coefficients for the zero function (cf.~\eqref{eqn:fourier-laplace-local-connection}), which implies that $L_{n}^{m}=(L')_{n}^{m}$ for all $n$ and $m$. Uniqueness of the local coefficients implies that the local expansion of a potential obtained from applying the addition theorem for the Laplace kernel~\eqref{eqn:laplace-addition-thm} or solid harmonics~(\ref{eqn:rnm-addition-thm},~\ref{eqn:inm-addition-thm}) is the same as the local expansion defined in this section.

Two key analytical properties follow from the integral form of the local expansion. We shall make use of these properties later to justify interchanging local expansions and series.
\begin{lemma}[Uniform convergence of local expansions]%
\label{lem:locexp-uniform-convergence}%
The local expansion $\local{\bm{c}}{p}[f]: \bar{B}(\bm{c}, \rho) \to \mathbb{C}$ converges uniformly to $f$ as $p \to \infty$ on $\bar{B}(\bm{c}, r)$, for any $0 \leq r < \rho$.
\end{lemma}
\begin{proof}
From the Poisson integral formula, the expansion must converge to $f$ in the interior.
An integral estimate on the $n$-th term of the local expansion~\eqref{eqn:poisson-integral-partial-expansion} bounds this term from above by $\left(2n+1\right)\left(r/\rho\right)^n \bignorm{f|_{\bar{B}(\bm{c},\rho)}}_\infty$,
which implies the series converges uniformly. Note that there are generally no convergence guarantees for the boundary in the uniform norm.
\end{proof}

\begin{lemma}[Local expansions of uniformly convergent sequences of functions]%
\label{lem:locexp-uniform-limit}%
If $\{f_n: \bar{B}(\bm{c}, \rho) \to \mathbb{C} \mid n \in \mathbb{N}_0\}$ is a sequence of functions harmonic in $B(\bm{c}, \rho)$, continuous on $\bar{B}(\bm{c}, \rho)$, converging uniformly to a function $f: \bar{B}(\bm{c}, \rho) \to \mathbb{C}$ harmonic in $B(\bm{c}, \rho)$, then, fixing $p \in \mathbb{N}_0$, we have $\local{\bm{c}}{p}[f](\bm{x}) = \lim_{n \to \infty} \local{\bm{c}}{p}[f_n](\bm{x})$ for all $\bm{x} \in \bar{B}(\bm{c}, \rho)$.
\end{lemma}
\begin{proof}
The local expansion $\local{\bm{c}}{p}[\cdot]$ may be written as an integral operator with a continuous kernel (cf.~\eqref{eqn:local-coeff}). After a change of variable, we may take this integral to be over the sphere of radius $\rho$ centered at $\bm{c}$. The result follows as we may interchange integration and uniform limits.
\end{proof}

\subsubsection{Multipole expansions}
The multipole expansion of a function may also be defined based on the Poisson integral identity through a geometric inversion. Let $g: \mathbb{R}^3 \setminus B(\bm{0}, 1) \to \mathbb{C}$ be harmonic in the exterior of the unit ball and continuous on $\mathbb{R}^3 \setminus B(\bm{0}, 1)$. Assume $\norm{g(\bm{x})} \to 0$ as $\bm{x} \to \infty$. For $\bm{x} \in \mathbb{R}^3 \setminus \{\bm{0}\}$, define
\[
\bm{x}^* = \left.\bm{\hat{x}}\middle/{\norm{\bm{x}}}\right..
\]
The map $\bm{x} \mapsto \bm{x}^*$ is an inversion of $\bm{x}$ with respect to the unit sphere, mapping the interior component of the unit sphere minus the origin to the exterior component in a one-to-one fashion, and vice versa.
It can be shown~\cite[Ch.~4]{axler:harmonic-functions:2001} that the inverted function
\[
g^*(\bm{x}^*) = \norm{\bm{x}^*}^{-1} g(\bm{x}), \quad 0 < \norm{\bm{x}^*} \leq 1,
\]
with the removal of a removable singularity at the origin, extends to a function that is harmonic in the unit ball.

If $f$ is harmonic in the exterior of a sphere of radius $\rho>0$ centered at $\bm{c} \in \mathbb{R}^3$ and continuous on $\mathbb{R}^3 \setminus B(\bm{c}, \rho)$, and if $\norm{f(\bm{x})} \to 0$ as $\norm{\bm{x}} \to \infty$, then by considering the the Poisson integral representation of $\mathcal{T}_{\bm{c},\rho}[f]^*$, we arrive at an integral representation of $f$ given by
\begin{align*}
f(\bm{x}) &= \norm{\left(\frac{\bm{x}-\bm{c}}{\rho}\right)^*} \mathcal{T}_{\bm{c},\rho}[f]^*\left(\left[\frac{\bm{x}-\bm{c}}{\rho}\right]^*\right)\\
&=
\frac{\rho}{\norm{\bm{x}-\bm{c}}}
\int_{\mathbb{S}^2}
\mathcal{P}\left(\frac{\rho}{\norm{\bm{x}-\bm{c}}},
\widehat{\bm{x}-\bm{c}} \cdot \bm{\xi}\right)
f(\bm{\bm{c}} + \rho \bm{\xi}) \, dS(\bm{\xi}).
\end{align*}
Proceeding in a similar manner to the local expansion case, we
can obtain a representation for $f$ in a series of spherical harmonics known as the \emph{multipole expansion}. We summarize this in the next definition.
\begin{definition}[Multipole expansion]%
Let $\rho > 0$ and $\bm{c} \in \mathbb{R}^3$. Let $f: \mathbb{R}^3 \setminus B(\bm{c}, \rho) \to \mathbb{C}$ be harmonic in the exterior of $B(\bm{c}, \rho)$ and continuous on $\mathbb{R}^3 \setminus B(\bm{c}, \rho)$. Define the multipole coefficients of $f$ via the integrals
\begin{equation}
M_n^m = \rho^{n+1} \int_{\mathbb{S}^2} f(\bm{c}+\rho{\bm{\xi}}) \overline{Y_n^m(\bm{\xi})} \, dS(\bm{\xi}).
\end{equation}
The function
\begin{equation}
    \mpole{\bm{c}}{p}[f](\bm{x}) = \sum_{n=0}^p
    \sum_{m=-n}^n M_n^m I_n^m(\bm{x} - \bm{c}),
    \quad \bm{x} \in \mathbb{R}^3 \setminus B(\bm{c}, \rho),
\end{equation}
is called the \emph{$p$-th order multipole expansion} of $f$ centered at $\bm{c}$.
\end{definition}
\begin{remark}%
\label{rem:multipole-analogue}%
Statements directly analogous to Remark~\ref{rem:fl} and Lemmas~\ref{lem:locexp-uniform-convergence}~and~\ref{lem:locexp-uniform-limit} continue to hold in the multipole case.
\end{remark}

\subsubsection{Translation operators}

The FMM relies on the ability to shift the center of expansions. This is accomplished through a \emph{translation operator}, which in this paper is denoted via repeated composition of the operators $\local{\bm{c}}{p}[\cdot]$ and $\mpole{\bm{c}}{p}[\cdot]$. Computationally, the translation operators in this paper are the mathematically the same up to a change of basis as the original Greengard--Rokhlin analytical translation operators~\cite{greengard:1988:thesis} and the `point-and-shoot' variants that optimize the original operators using a rotation of the coordinate system~\cite{white:point-and-shoot:1996}. Other kinds of translations, such as those based on plane wave expansions, though accomplishing the same purpose, are generally not mathematically equivalent.

We shall make use of the following well-known property of translation operators, which says that harmonic potentials undergo a sequence of translations may in some cases be treated as if intermediate translations were omitted.
\begin{lemma}[Omitting intermediate translations]%
\label{lem:translation-operator-idempotence}%
Let $\rho > 0$ and let $\bm{c'}, \bm{c} \in \mathbb{R}^3$ be expansion centers. Let $p', p \geq 0$ be integers. Consider a local expansion $\local{\bm{c}}{p}[f]: \bar{B}(\bm{c}, \rho) \to \mathbb{C}$ and a multipole expansion $\mpole{\bm{c}}{p}[g]: \mathbb{R}^3 \setminus B(\bm{c}, \rho) \to \mathbb{C}$.
\begin{enumerate}[(a)]
\item \label{lem:translation-operator-idempotence:local} If $p' \geq p$ and $\bm{c'} \in B(\bm{c}, \rho)$, then
$\local{\bm{c'}}{p'}[\local{\bm{c}}{p}[f]] = \local{\bm{c}}{p}[f]$ on $\bar{B}(\bm{c'}, \rho - \norm{\bm{c}-\bm{c'}})$.
\item \label{lem:translation-operator-idempotence:mpole} If $p' \leq p$, then
$\mpole{\bm{c'}}{p'}[\mpole{\bm{c}}{p}[g]] = \mpole{\bm{c'}}{p'}[g]$ on $\mathbb{R}^3 \setminus B(\bm{c'}, \rho + \norm{\bm{c}-\bm{c'}})$.
\end{enumerate}
\end{lemma}
\begin{proof} As this result is standard, we only state its proof in abbreviated form. Let $k \geq 0$ be an integer.
If $p' \geq p$, from the addition theorem for regular solid harmonics, we observe that $\local{\bm{c'}}{p'+k}[\local{\bm{c}}{p}[f]] -
\local{\bm{c'}}{p'}[\local{\bm{c}}{p}[f]] = 0$. Letting $k \to \infty$, we obtain~(\ref{lem:translation-operator-idempotence:local}). If $p' \leq p$, from the addition theorem for irregular solid harmonics, it follows that $\mpole{\bm{c'}}{p'}[\mpole{\bm{c}}{p+k}[g] - \mpole{\bm{c}}{p}[g]] = 0$. Letting $k \to \infty$ implies $\mpole{\bm{c'}}{p'}[g - \mpole{\bm{c}}{p}[g]] = 0$, using the fact that we may interchange expansions and uniform limits, cf.~Lemma~\ref{lem:locexp-uniform-limit}. (\ref{lem:translation-operator-idempotence:mpole}) follows.
\end{proof}

\section{Approximation of local expansions}%
\label{sec:approximation-results}

This section presents error bounds for the approximation of local expansions. In GIGAQBX, one encounters the following abstract evaluation scenario when forming QBX local expansions. Consider
a source point located at $\bm{s} \in \mathbb{R}^3$ and an expansion
center at $\bm{c} \in \mathbb{R}^3$ with $\norm{\bm{s} - \bm{c}} > \rho$.
Let $\phi_{\bm{s}}$ be the potential due to $\bm{s}$ and suppose we have an approximation (such as a multipole expansion) to this point potential $\tilde{\phi}_{\bm{s}}$, valid in $\bar{B}(\bm{c}, \rho)$. We are interested in the accuracy of the approximation to the local expansion $\local{\bm{c}}{q}[\phi_{\bm{s}}]$ that can be attained by using $\local{\bm{c}}{q}[\tilde{\phi}_{\bm{s}}]$, i.e., the quantity
\[ \bignorm{\local{\bm{c}}{q}[\phi_{\bm{s}}] - \local{\bm{c}}{q}[{\tilde{\phi}}_{\bm{s}}]}_\infty \]
on $\bar{B}(\bm{c}, \rho)$. Furthermore, as is typical in FMM calculations, we often have an estimate for the `point' error $\bignorm{\phi_{\bm{s}} - \tilde{\phi}_{\bm{s}}}_\infty$. Ideally, we would like to compare the two errors.

In the first part of this section, we consider the case that the approximation
$\tilde{\phi}_{\bm{s}}$ is an arbitrary harmonic function. Though having the advantage of being generic, the bound derived in the first part is often conservative. In the second part, we are concerned with the special case that the approximation is obtained from a sequence of multipole/local translation operators, and derive a more precise bound.


\subsection{Generic error bounds}
\label{sec:approximation-theory-results}

Suppose that we are given an approximation to a (harmonic) potential that is itself harmonic. (Some examples of
harmonic approximations include multipole/local expansions, plane wave expansions~\cite{greengard_rokhlin_1997}, or linear combinations of fundamental solutions from point sources arranged on a sphere~\cite{anderson:fmm:1992,ying:kifmm:2004}.) A bound on the accuracy of the approximation to the local expansion may be given with the help of the following result from approximation theory.

\begin{proposition}[Norm of the Fourier--Laplace projection]%
\label{prop:fourier-laplace-norm}%
Let $f: \mathbb{S}^2 \to \mathbb{C}$ be continuous. For each $p \in \mathbb{N}_0$, a constant $\Lambda_p > 0$ independent of $f$ exists such that
\begin{equation}
\bignorm{\mathcal{F}_p f}_{\infty} \leq \Lambda_p \bignorm{f}_{\infty}.
\end{equation}
The constant $\Lambda_p$ satisfies
\begin{equation} \Lambda_p = \sqrt{\frac{8p}{\pi}} + o(\sqrt{p}). \end{equation}
\end{proposition}

That $\mathcal{F}_p$ is a bounded operator on $C(\mathbb{S}^2)$ is evident from writing it as an integral (using~\eqref{eqn:sph-harm-addition-thm})
\[
\mathcal{F}_p f(\bm{x}) = \int_{\mathbb{S}^2} \left(
\sum_{n=0}^p \frac{2n+1}{4\pi} P_n(\bm{\xi} \cdot \bm{x})\right) f(\bm{\xi}) \, dS(\bm{\xi}),
\quad \bm{x} \in \mathbb{S}^2,
\]
and observing that the integral kernel is a continuous function. The norm (or \emph{Lebesgue constant}) of the operator is the $L^1$ norm of the above kernel. The full computation of this norm is given in~\cite{gronwall:1914:laplace-series}.

Proposition~\ref{prop:fourier-laplace-norm} implies the following.
\begin{lemma}[Bound on local expansion growth]%
\label{lem:locexp-approx-theory-bound}%
Consider a local expansion of a function $f: \bar{B}(\bm{c}, \rho) \to \mathbb{C}$ harmonic in $B(\bm{c}, \rho)$ and continuous on $\bar{B}(\bm{c}, \rho)$. Then
\begin{equation}
\bignorm{\local{\bm{c}}{p}[f]}_{\infty} \leq
\Lambda_p \bignorm{f}_\infty.
\end{equation}
\end{lemma}
\begin{proof}
For $\bm{x}$ with $\norm{\bm{x} - \bm{c}} = r \leq \rho$, (cf.~\eqref{eqn:fourier-laplace-local-connection})
\[
\local{\bm{c}}{p}[f](\bm{x}) = \fourier{p}
\left[\mathcal{T}_{\bm{c},r}
\left[
\left.
f
\right|_{\{\bm{y}: \norm{\bm{y}-\bm{c}} = r\}}
\right]
\right]
(\widehat{\bm{x}-\bm{c}}),
\]
Thus by Proposition~\ref{prop:fourier-laplace-norm} and the definition of~$\mathcal{T}_{\bm{c},r}$~\eqref{eqn:translated-scaled-function},
\[
\norm{\local{\bm{c}}{p}[f](\bm{x})} \leq
\Lambda_p \bignorm{\left. f \right|_{\{\bm{y}: \norm{\bm{y} - \bm{c}}=r\}}}_\infty.
\]
The result follows by taking the maximum over all $0 \leq r \leq \rho$.
\end{proof}

This lemma establishes that the error in approximation of a local expansion $\local{\bm{c}}{q}[\phi_{\bm{s}}]$ can be no worse than an order-dependent constant times the error in approximating the original potential:
\begin{equation}
\label{eqn:generic-local-approximation}
\bignorm{\local{\bm{c}}{q}[\phi_{\bm{s}}] - \local{\bm{c}}{q}[{\tilde{\phi}}_{\bm{s}}]}_\infty \leq \Lambda_q \bignorm{\phi_{\bm{s}} - \tilde{\phi}_{\bm{s}}}_\infty.
\end{equation}
For fixed order $q$, this says that as $\tilde{\phi}_{\bm{s}} \to \phi_{\bm{s}}$, we can expect the approximation to the local expansion to become proportionally more accurate.

\begin{remark}[Generalization to other kernels and dimensions]%
\label{rem:generalization}%
We have not considered local expansions of non-harmonic functions in this paper, but it is worth noting that Lemma~\ref{lem:locexp-approx-theory-bound} is not specific to the Laplace PDE. It holds true for local expansions of \emph{any PDE} that can be reformulated as a Fourier--Laplace series. For instance, for the Helmholtz kernel with parameter $k > 0$, and $\norm{\bm{x}} < \norm{\bm{y}}$, one has the addition theorem
\[
\frac{e^{ik\norm{\bm{x}-\bm{y}}}}{\norm{\bm{x}-\bm{y}}}
=
ik \sum_{n=0}^\infty (2n+1) j_n(k \norm{\bm{x}})
h_n(k \norm{\bm{y}})
P_n(\bm{\hat{x}} \cdot \bm{\hat{y}}),
\]
where the functions $j_n$ and $h_n$ are, respectively, spherical Bessel and Hankel functions of the first kind (cf.~\cite[eq.~(10.60.1,~10.60.2)]{nist:dlmf}). After applying the spherical harmonic addition theorem~\eqref{eqn:sph-harm-addition-thm}, one obtains a local expansion of the Helmholtz kernel as a series in spherical harmonics. It is clear that a direct analogue of Remark~\ref{rem:fl} holds for this local expansion.

In two dimensions, similar remarks apply. The main difference is that the orthonormal expansion of a function on the unit circle is expressed as a Fourier series. The Lebesgue constant of the Fourier projection is for the two-dimensional case is $\Lambda_p = 4/\pi^2 \log p + O(1)$~\cite[Lem.~2.2]{rivlin:1981:approximation-theory}.
\end{remark}

The main issue with~\eqref{eqn:generic-local-approximation} is that as $q \to \infty$ it implies a bound increasingly worse compared to the point error. The reason for this is that we have only made use of continuity of $\tilde{\phi}_{\bm{s}}$ on the boundary, which is not even sufficient to guarantee that the expansion of $\tilde{\phi}_{\bm{s}}$ converges there. The approximation $\tilde{\phi}_{\bm{s}}$ is often a smooth function and therefore should obey a better bound. In the next section, we analyze a common situation in which the approximation comes from a sequence of multipole/local translations. In that case, we show that one can replace $\Lambda_q$ in~\eqref{eqn:generic-local-approximation} with the constant~$1$.

\subsection{Error bounds for Greengard--Rokhlin-style translation operators}
The bounds in this section consider the case that our approximation to the potential, $\tilde{\phi}_{\bm{s}}$, is formed using a  three-dimensional Laplace FMM making use of the original Greengard--Rokhlin translation operators or their rotation-based variant. Without loss of generality we may restrict our attention to those translations found at a single level of the hierarchical tree structure in the FMM, because intermediate translations that cross levels in the hierarchy do not change the value of the expansion (see Lemma~\ref{lem:translation-operator-idempotence}). This leaves three kinds of translations. Here and in the rest of this section, $p$ represents an intermediate translation order and $q$ represents a final expansion order.
\paragraph{Source $\to$ Local($p$) $\to$ Local($q$)} This corresponds to a \emph{List~4} interaction of a typical FMM, or \emph{List~4~close} of GIGAQBX\@. This case is covered in Theorem~\ref{thm:l2qbxl}.
For the definition of the interaction lists in the canonical adaptive FMM, see~\cite[Sec.~3.2]{carrier:1988:adaptive-fmm}, and for the ones in the GIGAQBX FMM, see~\cite[Sec.~3.2]{wala:2020:gigaqbxts}.
\paragraph{Source $\to$ Multipole($p$) $\to$ Local($q$)} This case corresponds to a \emph{List~3} interaction of a typical FMM (\emph{List~3~close} of GIGAQBX), and is covered in Theorem~\ref{thm:m2qbxl}.
\paragraph{Source $\to$ Multipole($p$) $\to$ Local($p$) $\to$ Local($q$)} This corresponds to \emph{List~2} interaction of the FMM\@. This is covered in Theorem~\ref{thm:m2l2qbxl}.

\subsubsection{Preliminaries}

The following three lemmas are used later. Lemma~\ref{prop:binom-product-inequality} gives combinatorial estimates. Lemmas~\ref{lem:rn0-local} and~\ref{lem:in0-local} bound the growth of local expansions of solid harmonics.

\begin{lemma}%
\label{prop:binom-product-inequality}%
Let $n, m, k \in \mathbb{N}_0$.
\begin{enumerate}[(a)]
\item If $n \geq k$, then
\[\binom{n+k}{m} \binom{n-k}{m} \leq \binom{n}{m}^2.\]
\item If $m \geq k$, then
\[\binom{n}{m+k} \binom{n}{m-k} \leq \binom{n}{m}^2.\]
\end{enumerate}
\end{lemma}
\begin{proof}
This is obvious when $m = 0$ or $m \geq n$.
When $0 < m < n$, this follows by induction on $k$.
\end{proof}

\begin{lemma}[Bound on local expansion of $R_n^0$]%
\label{lem:rn0-local}%
Let $n, p \geq 0$ be integers and let $\bm{c}, \bm{t} \in \mathbb{R}^3$. Then
\begin{equation}
\norm{\local{\bm{c}}{p}[R_n^0](\bm{t})} \leq \sqrt{\frac{2n+1}{4\pi}}
\left( \norm{\bm{c}} + \norm{\bm{t} - \bm{c}} \right)^n.
\end{equation}
\end{lemma}

\begin{proof}
To obtain the local expansion of $R_n^0$, apply the addition theorem for solid harmonics~\eqref{eqn:rnm-addition-thm} to the vectors $\bm{x} = \bm{t} - \bm{c}$ and $\bm{y} = \bm{c}$. Define the terms of the local expansion
\[
S^R_{n, \nu}(\bm{c}, \bm{t}) =
  A_n^0 \sum_{\ilocloc{\mu}{\nu}{0}{n-\nu}}
  \binom{n}{\nu+\mu}
  \frac{R_{n-\nu}^{-\mu}(\bm{c}) R_{\nu}^{\mu}(\bm{t}-\bm{c})}{A^{-\mu}_{n-\nu}A_{\nu}^{\mu}},
  \quad \nu \in \{0, \ldots, n\},
\]
so that
\begin{equation}
  \label{eqn:local-rn0-v2}
  \local{\bm{c}}{p}[R_n^0](\bm{t})
  = \sum_{\nu=0}^{\min(p, n)} S^R_{n,\nu}(\bm{c}, \bm{t}).
\end{equation}

After expanding the combinatorial, normalization, and spherical harmonic factors in
$S^R_{n,\nu}(\bm{c}, \bm{t})$, it can be written as
\[
S^R_{n,\nu}(\bm{t},\bm{c})
=
A_n^0 \sum_{\ilocloc{\mu}{\nu}{0}{n-\nu}}
\alpha^{\mu}_{n,\nu}(\bm{c})
\beta^{\mu}_{n,\nu}(\bm{t} - \bm{c}),
\quad \nu \in \{0, \ldots, n\},
\]
where
\begin{align}
    \alpha_{n,\nu}^{\mu}(\bm{x})
    &=
    \sqrt{ \frac{4\pi}{2(n-\nu)+1} \binom{n}{\nu-\mu} \binom{n}{\nu+\mu} }
    \norm{\bm{x}}^{n-\nu}
    Y^{-\mu}_{n-\nu}(\bm{\hat{x}}), \\
    \beta_{n,\nu}^{\mu}(\bm{x})
    &=
    \sqrt{ \frac{4\pi}{2\nu+1}}
    \norm{\bm{x}}^{\nu}
    Y_{\nu}^{\mu}(\bm{\hat{x}}).
\end{align}
(To obtain this, first expand the combinatorial and normalization terms
in~$S^R_{n,\nu}(\bm{c}, \bm{t})$, yielding as an intermediate step
\[
\frac{\binom{n}{\nu+\mu}}{A_{n-\nu}^{-\mu} A_{\nu}^{\mu}}
=
\sqrt{\frac{(4\pi)^2}{(2(n-\nu)+1)(2\nu+1)} \frac{n!^2}{(\nu-\mu)!(\nu+\mu)!(n-\nu+\mu)!(n-\nu-\mu)!}}.
\]
We omit intermediate calculations of this form in the remainder of this paper.)

Via the Cauchy--Schwarz inequality
\begin{equation}
\label{eqn:snnu-cauchy-schwartz}
\norm{S^R_{n,\nu}(\bm{c},\bm{t})}^2
\leq
(A_n^0)^2
\left(
\sum_{\ilocloc{\mu}{\nu}{0}{n-\nu}}
\norm{\alpha^{\mu}_{n,\nu}(\bm{c})}^2
\right)
\left(
\sum_{\ilocloc{\mu}{\nu}{0}{n-\nu}}
\norm{\beta^{\mu}_{n,\nu}(\bm{t}-\bm{c})}^2
\right)
.
\end{equation}
We can bound the second term on the right-hand side as
\begin{align}
    \label{eqn:alphannumu-sum-sq-bound}
    \sum_{\ilocloc{\mu}{\nu}{0}{n-\nu}}
    \norm{\alpha_{n,\nu}^{\mu}(\bm{x})}^2 &=
    \left(
    \sqrt{\frac{4\pi}{2(n-\nu)+1}}
    \norm{\bm{x}}^{(n-\nu)}
    \right)^2 \\
    &\quad\quad
    \sum_{\mu = \max(-\nu, -(n-\nu))}^{\min(\nu, n-\nu)}
    \binom{n}{\nu-\mu}
    \binom{n}{\nu+\mu}
    \norm{Y_{n-\nu}^{-\mu}(\bm{\hat{x}})}^2 \nonumber \\
    &\leq
    \left[ \sqrt{\frac{4\pi}{2(n-\nu)+1}}
    \norm{\bm{x}}^{n-\nu}
    \binom{n}{\nu}\right]^2
    \sum_{\mu=-(n-\nu)}^{n-\nu} \norm{Y_{n-\nu}^{-\mu}(\bm{\hat{x}})}^2 \label{eqn:alphannumu-sum-sq-bound-step1} \\
    &=
    \left[
    \binom{n}{\nu}
    \norm{\bm{x}}^{n-\nu}
    \right]^2, \label{eqn:alphannumu-sum-sq-bound-step2}
\end{align}
where we have used Lemma~\ref{prop:binom-product-inequality} in~\eqref{eqn:alphannumu-sum-sq-bound-step1}, and~\eqref{eqn:sph-harm-addition-corollary} in~\eqref{eqn:alphannumu-sum-sq-bound-step2}. The third term on the right-hand side above is bounded as
\begin{align}
    \label{eqn:betannumu-sum-sq-bound}
    \sum_{\ilocloc{\mu}{\nu}{0}{n-\nu}}
    \norm{\beta_{n,\nu}^{\mu}(\bm{x})}^2 &=
    \left(
    \sqrt{\frac{4\pi}{2\nu+1}}
    \norm{\bm{x}}^{\nu}
    \right)^2
    \sum_{\mu = \max(-\nu, -(n-\nu))}^{\min(\nu, n-\nu)}
    \norm{Y_{\nu}^{\mu}(\bm{\hat{x}})}^2 \\
    &\leq
    \left(
    \sqrt{\frac{4\pi}{2\nu+1}}
    \norm{\bm{x}}^{\nu}
    \right)^2
    \sum_{\mu=-\nu}^{\nu}
    \norm{Y_{\nu}^{\mu}(\bm{\hat{x}})}^2 \nonumber \\
    &= \norm{\bm{x}}^{2\nu}. \nonumber
\end{align}
Combining~\eqref{eqn:snnu-cauchy-schwartz},~\eqref{eqn:alphannumu-sum-sq-bound}, and~\eqref{eqn:betannumu-sum-sq-bound},
\begin{equation}
\label{eqn:snnu-bound}
\norm{S^R_{n,\nu}(\bm{c}, \bm{t})} \leq
\sqrt{\frac{2n+1}{4\pi}}
\binom{n}{\nu} \norm{\bm{c}}^{n-\nu} \norm{\bm{t} - \bm{c}}^{\nu}.
\end{equation}
Applying~\eqref{eqn:snnu-bound} to~\eqref{eqn:local-rn0-v2}, we obtain
\begin{align}
    \norm{\local{\bm{c}}{p}[R_n^0](\bm{t})}
    &\leq \sqrt{\frac{2n+1}{4\pi}} \sum_{\nu=0}^{\min(p,n)}
    \binom{n}{\nu} \norm{\bm{c}}^{n-\nu} \norm{\bm{t} - \bm{c}}^{\nu} \nonumber \\
    &\leq \sqrt{\frac{2n+1}{4\pi}} (
    \norm{\bm{c}} +  \norm{\bm{t} - \bm{c}})^n
     \nonumber.
\end{align}
\end{proof}

\begin{lemma}[Bound on local expansion of $I_n^0$]%
\label{lem:in0-local}%
Let $n, p \geq 0$ be integers and let $\bm{c}, \bm{t} \in \mathbb{R}^3$ with $\norm{\bm{t}-\bm{c}} < \norm{\bm{c}}$. Then
\[
\norm{\local{\bm{c}}{p}[I_n^0](\bm{t})} \leq \sqrt{\frac{2n+1}{4\pi}}
\left( \norm{\bm{c}} - \norm{\bm{t} - \bm{c}} \right)^{-(n+1)}.
\]
\end{lemma}

\begin{proof}
Applying the addition theorem~\eqref{eqn:inm-addition-thm} to the vectors $\bm{x} = \bm{t} - \bm{c}$ and $\bm{y} = \bm{c}$, we obtain the local expansion of $I_n^0$, as follows. Define the terms of the expansion
\begin{equation}
\label{eqn:in0-snnu}
S^I_{n,\nu}(\bm{c}, \bm{t}) = A_n^0 \sum_{\mu \in I(0,\nu)}
(-1)^{\nu-\mu}
\binom{n+\nu-\mu}{n}
\frac{
I^{\mu}_{n+\nu}(\bm{c})
R^{-\mu}_{\nu}(\bm{t} - \bm{c})
}{
A^{\mu}_{n+\nu}
A^{-\mu}_{\nu}
},
\quad \nu \in \mathbb{N}_0,
\end{equation}
so that
\begin{equation}
\label{eqn:in0-local-exp}
\local{\bm{c}}{p}[I_n^0](\bm{t})
= \sum_{\nu=0}^p S^I_{n,\nu}(\bm{c},\bm{t}).
\end{equation}
Expanding out and simplifying~\eqref{eqn:in0-snnu}, we write $S^I_{n,\nu}(\bm{c}, \bm{t})$ as
\begin{equation}
\label{eqn:in0-snnu-simplified}
S^I_{n,\nu}(\bm{c},\bm{t})
=
A_n^0 \sum_{\mu\in I(0,\nu)}
(-1)^{\nu-\mu}
\alpha^{\mu}_{n,\nu}(\bm{c})
\beta^{\mu}_{n,\nu}(\bm{t} - \bm{c})
\end{equation}
where
\begin{align}
    \alpha_{n,\nu}^{\mu}(\bm{x})
    &=
    \sqrt{\frac{4\pi}{2(n+\nu)+1}
    \binom{n+\nu-\mu}{n} \binom{n+\nu+\mu}{n}}
    \norm{\bm{x}}^{-(n+\nu+1)}
    Y^{\mu}_{n+\nu}(\bm{\hat{x}}), \\
    \beta_{n,\nu}^{\mu}(\bm{x})
    &=
    \sqrt{ \frac{4\pi}{2\nu+1}}
    \norm{\bm{x}}^{\nu}
    Y_{\nu}^{-\mu}(\bm{\hat{x}}).
\end{align}
A calculation similar to the proof of the previous
lemma shows
\begin{equation}
\label{eqn:in0-snnu-term-bounds}
\sum_{\mu \in I(0,\nu)}
\norm{\alpha_{n,\nu}^{\mu}(\bm{x})}^2
\leq \left[\binom{n+\nu}{n} \norm{\bm{x}}^{-(n+\nu+1)}\right]^2,
\quad
\sum_{\mu \in I(0,\nu)}
\norm{\beta_{n,\nu}^{\mu}(\bm{x})}^2
\leq \norm{\bm{x}}^{2\nu},
\end{equation}
and applying the Cauchy--Schwarz inequality to~\eqref{eqn:in0-snnu-simplified}
using~\eqref{eqn:in0-snnu-term-bounds} yields
\begin{equation}
\label{eqn:in0-snnu-bound}
\norm{S^I_{n,\nu}(\bm{c}, \bm{t})} \leq
\sqrt{\frac{2n+1}{4\pi}} \binom{n+\nu}{n} \norm{\bm{c}}^{-(n+\nu+1)} \norm{\bm{t} - \bm{c}}^{\nu}.
\end{equation}
Therefore, from~\eqref{eqn:in0-local-exp},
\begin{align}
\norm{\local{\bm{c}}{p}[I_n^0](\bm{t})}
& \leq \sqrt{\frac{2n+1}{4\pi}}
\sum_{\nu=0}^p \binom{n+\nu}{n}
\norm{\bm{c}}^{-(n+\nu+1)} \norm{\bm{t} - \bm{c}}^{\nu} \nonumber \\
& \leq \sqrt{\frac{2n+1}{4\pi}}
\norm{\bm{c}}^{-(n+1)}
\sum_{\nu=0}^\infty \binom{n+\nu}{n}
\left(\frac{\norm{\bm{t} - \bm{c}}}{\norm{\bm{c}}}\right)^{\nu} \label{eqn:binomial-series-pt1}  \\
&=
\sqrt{\frac{2n+1}{4\pi}}
\left(\norm{\bm{c}} - \norm{\bm{t}-\bm{c}}\right)^{-(n+1)}.\label{eqn:binomial-series-pt2}
\end{align}
To obtain~\eqref{eqn:binomial-series-pt2} from~\eqref{eqn:binomial-series-pt1}
we used a variation of the binomial series~\cite[eq.~(4.6.7)]{nist:dlmf} for
negative exponents.
\end{proof}

\begin{figure}
\begin{minipage}[t]{0.48\textwidth}
  \centering
  \begin{tikzpicture}[scale=1]
    \node [disk, minimum size=3cm, color=localcolor](Gamma) at (0,0) {};
    \draw [fill] (Gamma) circle (1pt) node [below] {$\bm{0}$};
    \draw [<->] (Gamma.center) -- (Gamma.north) node [right, midway] {$r$};

    \path (Gamma) ++ (-45:0.9cm) node [disk, color=qbxcolor, minimum size=1cm](lexp) {};
    \path (lexp) ++ (20:0.5cm) node (y) {};
    \draw [fill] (lexp) circle (1pt) node [below] {$\bm{c}$};
    \draw [fill] (y) circle (1pt) node [right] {$\bm{t}$};

    \path (Gamma) ++ (180:4cm) node [draw=none, fill=none, inner sep=0pt, outer sep=0pt](gamma) {};
    \draw [fill] (gamma) circle (1pt) node [left] {$\bm{s}$};
    \draw [<->] (gamma) -- (Gamma.center) node [below, midway] {$R$};

    \draw [localcolor]
        (gamma)
        edge [bend left, ->, dashed]
        node [maskedbg, above, yshift=0.1cm, midway] {$\local{\bm{0}}{p}[\pot{\bm{s}}]$}
        (Gamma.center);

    \draw [qbxcolor]
        (Gamma.center)
        edge[bend left, ->, dashed]
        node [maskedbg, right, xshift=-0.2cm, yshift=0.4cm]
        {$\local{\bm{c}}{q}[\local{\bm{0}}{p}[\pot{\bm{s}}]]$}
        (lexp.center);

    \draw []
        (gamma)
        edge[bend right, ->, dashed]
        node[below, midway] {$\local{\bm{c}}{q}[\pot{\bm{s}}]$}
        (lexp.center);
  \end{tikzpicture}
  \caption{%
    Obtaining the local expansion of a point potential using an intermediate
    local expansion.  This provides the setting for Theorem~\ref{thm:l2qbxl}.
  }%
  \label{fig:l2qbxl}%
\end{minipage}
\hfill
\begin{minipage}[t]{0.48\textwidth}
  \centering
  \begin{tikzpicture}[scale=1]

    \path (0,0) node [disk, color=qbxcolor, minimum size=1cm](lexp) {};
    \path (lexp) ++ (20:0.5cm) node (y) {};
    \draw [fill] (lexp) circle (1pt) node [right] {$\bm{c}$};
    \draw [fill] (y) circle (1pt) node [right] {$\bm{t}$};

    \path (Gamma) ++ (-4cm, 0) node [disk, color=srccolor, minimum size=1.5cm](gamma) {};
    \path (gamma) ++(155:0.75cm) coordinate (source);
    \draw [fill] (gamma) circle (1pt) node [below] {$\bm{0}$};
    \draw [fill] (source) circle (1pt) node [left] {$\bm{s}$};

    \draw [color=srccolor]
        (source)
        edge [bend left, dashed, ->]
        node [midway, above, yshift=0.5cm] {$\mpole{\bm{0}}{p}[\pot{\bm{s}}]$}
        (gamma.center);

    \draw [<->] (gamma.center) -- (lexp) node [below, midway] {$R$};

    \draw[loosely dashed, color=srccolor]
         (gamma.center) ++ (-30:3.5cm) arc[start angle=-30, end angle=30, radius=3.5cm];

    \draw [<->] (gamma.center) -- (gamma.north)
        node [right, midway] {$r$};

    \draw [color=qbxcolor]
        (gamma.center)
        edge [bend left, dashed, ->]
        node [midway, above]
        {$\local{\bm{c}}{q}[\mpole{\bm{0}}{p}[\pot{\bm{s}}]]$}
        (lexp.center);

    \draw (source)
        edge [bend right=65, dashed, ->]
        node [midway, below] {$\local{\bm{c}}{q}[\pot{\bm{s}}]$}
        (lexp.center);
  \end{tikzpicture}
  \caption{%
    Obtaining the local expansion of a point potential using an
    intermediate multipole expansion. This provides the geometric setting for
    Theorem~\ref{thm:m2qbxl}.
  }%
  \label{fig:m2qbxl}%
\end{minipage}
\hfill
\end{figure}

\begin{figure}
  \centering
  \begin{tikzpicture}[scale=1]
    \tikzset{%
    }
    \node [disk, minimum size=3cm, color=localcolor](Gamma) at (0,0) {};
    \draw [fill] (Gamma) circle (1pt) node [left] {$\bm{c}$};
    \draw [<->] (Gamma.center) -- (Gamma.north) node [right, midway] {$r'$};

    \path (Gamma) ++ (-45:0.9cm) node [disk, color=qbxcolor, minimum size=1cm](lexp) {};
    \path (lexp) ++ (20:0.5cm) node (y) {};
    \draw [fill] (lexp) circle (1pt) node [below] {$\bm{c'}$};
    \draw [fill] (y) circle (1pt) node [right] {$\bm{t}$};

    \path (Gamma) ++ (-4cm, 0) node [disk, color=srccolor, minimum size=1.5cm](gamma) {};
    \path (gamma) ++(155:0.75cm) coordinate (source);
    \draw [fill] (gamma) circle (1pt) node [below] {$\bm{0}$};
    \draw [fill] (source) circle (1pt) node [left] {$\bm{s}$};

    \draw [color=srccolor]
        (source)
        edge [bend left, dashed, ->]
        node [midway, above, yshift=0.5cm] {$\mpole{\bm{0}}{p}[\pot{\bm{s}}]$}
        (gamma.center);

    \draw [<->] (gamma.center) -- (gamma.north) node [right, midway] {$r$};

    \draw [<->] (gamma.center) -- (Gamma.west) node [above, midway] {$R$};

    \coordinate (Rpleft) at ([yshift=-2cm] gamma.east);
    \coordinate (Rpright) at ([yshift=-2cm] Gamma.center);
    \draw [dotted] (gamma.east) -- (Rpleft);
    \draw [dotted] (Gamma.center) -- (Rpright);
    \draw [<->] (Rpleft) -- (Rpright) node [below, midway] {$R'$};

    \draw[loosely dashed, color=srccolor]
         (gamma.center) ++ (-45:2.5cm) arc[start angle=-45, end angle=45, radius=2.5cm];

    \draw [color=localcolor]
        (gamma.center)
        edge [bend left=45, dashed, ->]
        node [maskedbg, midway, above]
        {$\local{\bm{c}}{p}[\mpole{\bm{0}}{p}[\pot{\bm{s}}]]$}
        (Gamma.center);

    \draw
        (source)
        edge [bend right=65, dashed, ->]
        node [maskedbg, midway, above, xshift=0.5cm] {$\local{\bm{c'}}{q}[\pot{\bm{s}}]$}
        (lexp.center);

    \draw [color=qbxcolor]
        (Gamma.center)
        edge[bend left, ->, dashed]
        node [maskedbg, right, xshift=-0.2cm, yshift=0.4cm]
        {$\local{\bm{c'}}{q}[\local{\bm{c}}{p}[\mpole{\bm{0}}{p}[\pot{\bm{s}}]]]$}
        (lexp.center);
  \end{tikzpicture}
  \caption{%
    Obtaining the local expansion of a point potential using intermediate
    multipole expansion and an intermediate local expansion. This provides the geometric setting for
    Theorem~\ref{thm:m2l2qbxl}.
  }
  \label{fig:m2l2qbxl}%
\end{figure}

\subsubsection{Bounds} We are now ready to state the main bounds in this section.

\begin{theorem}[{Source $\to$ Local($p$) $\to$ Local($q$), cf.~\cite[Hyp.~2]{wala:2019:gigaqbx3d}}]%
\label{thm:l2qbxl}%
Let $p, q \geq 0$ be integers, let $R > r > 0$, and let $\bm{s}, \bm{c}, \bm{t} \in \mathbb{R}^3$.
Assume $\norm{\bm{s}} \geq R$, $\norm{\bm{c}} \leq r$, and $\bm{t} \in \bar{B}(\bm{c}, r - \norm{\bm{c}})$.
Consider a translation sequence, depicted in Figure~\ref{fig:l2qbxl}, in which a $p$-th order local expansion of the potential due to $\bm{s}$ is formed at the origin and subsequently translated to a $q$-th order local expansion at $\bm{c}$, yielding $\local{\bm{c}}{q}[\local{\bm{0}}{p}[\pot{\bm{s}}]]$.
Then a bound for approximating
$\local{\bm{c}}{q}[\ptpot_{\bm{s}}]$ with
$\local{\bm{c}}{q}[\local{\bm{0}}{p}[\ptpot_{\bm{s}}]]$ is as follows:
\[
\norm{
  \local{\bm{c}}{q}[\ptpot_{\bm{s}}](\bm{t})
  -
  \local{\bm{c}}{q}[\local{\bm{0}}{p}[\ptpot_{\bm{s}}]](\bm{t})
  } \leq \frac{1}{R - r} \left(\frac{r}{R}\right)^{p+1}.
\]
\end{theorem}

\begin{proof}
  Without loss of generality, by a rotation of the coordinate system we may assume
  $\bm{\hat{s}} = \transpose{(0,0,1)}$.
  The expansion of the potential due to $\bm{s}$, for $\norm{\bm{x}} < \norm{\bm{s}}$, is
  (cf.~\eqref{eqn:laplace-addition-thm-solid})
  \[
    \ptpot_{\bm{s}}(\bm{x}) = \sum_{n=0}^\infty
    \frac{4\pi}{2n+1} \sum_{m=-n}^n R_n^m(\bm{x})
    \overline{I_n^m(\bm{s})}.
  \]
  Since $\bm{\hat{s}} = \transpose{(0, 0, 1)}$, we can simplify
  using~\eqref{eqn:sph-harm-z-axis}, and obtain the local expansion
  at the origin as follows:
  \[
    \local{\bm{0}}{p}[\ptpot_{\bm{s}}](\bm{x}) = \sum_{n=0}^p
    \sqrt{\frac{4\pi}{2n+1}} \norm{\bm{s}}^{-(n+1)} R_n^0(\bm{x}).
  \]
  From linearity of $\local{\bm{c}}{q}[\cdot]$,
  \[
    \local{\bm{c}}{q}[\ptpot_{\bm{s}}](\bm{t})
    -
    \local{\bm{c}}{q}[\local{\bm{0}}{p}[\ptpot_{\bm{s}}]](\bm{t})
    =
    \local{\bm{c}}{q}\left[
      \sum_{n=p+1}^\infty \sqrt{\frac{4\pi}{2n+1}} \norm{\bm{s}}^{-(n+1)} R_n^0(\cdot)
      \right](\bm{t}).
  \]
  In the right-hand side of the above expression, we can interchange the order of summation and expansion,
  as the summation converges uniformly in $n$ on $\bar{B}(\bm{c}, r)$ (cf.~Lemmas~\ref{lem:locexp-uniform-convergence} and~\ref{lem:locexp-uniform-limit}). Therefore,
  \begin{equation}
    \label{eqn:accel-error-series}
    \local{\bm{c}}{q}[\ptpot_{\bm{s}}](\bm{t})
    -
    \local{\bm{c}}{q}[\local{\bm{0}}{p}[\ptpot_{\bm{s}}]](\bm{t})
    =
    \sum_{n=p+1}^\infty \sqrt{\frac{4 \pi}{2n+1}} \norm{\bm{s}}^{-(n+1)} \local{\bm{c}}{q}[R_n^0](\bm{t}).
  \end{equation}
  From Lemma~\ref{lem:rn0-local},
\begin{equation}
  \label{eqn:rn0-bound}
  \norm{\local{\bm{c}}{q}[R_n^0](\bm{t})}
  \leq
  \sqrt{\frac{2n+1}{4\pi}} (\norm{\bm{c}} + \norm{\bm{t} - \bm{c}})^n
  \leq
  \sqrt{\frac{2n+1}{4\pi}} r^n.
\end{equation}
Using~\eqref{eqn:accel-error-series} and~\eqref{eqn:rn0-bound},
\begin{align*}
\norm{\local{\bm{c}}{q}[\ptpot_{\bm{s}}](\bm{t}) - \local{\bm{c}}{q}[\local{\bm{0}}{p}[\ptpot_{\bm{s}}]](\bm{t})}
&=
\norm{ \sum_{n=p+1}^\infty \sqrt{\frac{4 \pi}{2n+1}} \norm{\bm{s}}^{-(n+1)}
 \local{\bm{c}}{q}[R_n^0](\bm{t})}  \\
&\leq
 \sum_{n=p+1}^\infty r^n \norm{\bm{s}}^{-(n+1)}  \nonumber \\
&\leq \frac{1}{R} \sum_{n=p+1}^\infty \left(\frac{r}{R}\right)^n \nonumber \\
&= \frac{1}{R-r} \left(\frac{r}{R}\right)^{p+1} \nonumber.
\end{align*}

\end{proof}

\begin{theorem}[{(Source $\to$ Multipole($p$) $\to$ Local($q$), cf.~\cite[Hyp.~1]{wala:2019:gigaqbx3d}}]%
\label{thm:m2qbxl}%
Let $p, q \geq 0$ be integers, let $R > r > 0$, and let $\bm{s}, \bm{c}, \bm{t} \in \mathbb{R}^3$.
Suppose that $\norm{\bm{s}} \leq r$, $\norm{\bm{c}} \geq R$, and $t \in \bar{B}(\bm{c}, \norm{\bm{c}}-R)$.
Consider a sequence of translations, depicted in Figure~\ref{fig:m2qbxl}, in which a $p$-th order multipole expansion of the potential due to $\bm{s}$ is formed at the origin and subsequently translated to $q$-th order local expansion at $\bm{c}$, yielding $\local{\bm{c}}{q}[\mpole{\bm{0}}{p}[\pot{\bm{s}}]]$.
Then
a bound for approximating
$\local{\bm{c}}{q}[\ptpot_{\bm{s}}]$ with
$\local{\bm{c}}{q}[\mpole{\bm{0}}{p}[\ptpot_{\bm{s}}]]$ is as follows:
\[
\norm{
  \local{\bm{c}}{q}[\ptpot_{\bm{s}}](\bm{t})
  -
  \local{\bm{c}}{q}[\mpole{\bm{0}}{p}[\ptpot_{\bm{s}}]](\bm{t})
  } \leq \frac{1}{R - r} \left(\frac{r}{R}\right)^{p+1}.
\]
\end{theorem}

\begin{proof}
  Proceeding as with the proof of Theorem~\ref{thm:l2qbxl}, without loss of generality assume
  $\bm{\hat{s}} = \transpose{(0,0,1)}$.
  We write the expansion of the potential due to $\bm{s}$, with $\norm{\bm{x}} > \norm{\bm{s}}$,
  as (cf.~\eqref{eqn:laplace-addition-thm-solid})
  \[
    \ptpot_{\bm{s}}(\bm{x}) = \sum_{n=0}^\infty
    \frac{4\pi}{2n+1} \sum_{m=-n}^n \overline{R_n^m(\bm{s})}
    I^m_n(\bm{x}),
  \]
  observing we may take the conjugate of the inner summation in~\eqref{eqn:laplace-addition-thm-solid} since the value of the inner summation is real.
  Simplifying using~\eqref{eqn:sph-harm-z-axis}, we obtain the
  multipole expansion at the origin:
  \begin{equation}
  \label{eqn:m2qbxl-multipole}
  \mpole{\bm{0}}{p}[\ptpot_{\bm{s}}](\bm{x})
  = \sum_{n=0}^p \sqrt{\frac{4\pi}{2n+1}} \norm{\bm{s}}^n I_n^0(\bm{x}).
  \end{equation}
  Via linearity of $\local{\bm{c}}{q}[\cdot]$,
  \[
    \local{\bm{c}}{q}[\ptpot_{\bm{s}}](\bm{t})
    -
    \local{\bm{c}}{q}[\mpole{\bm{0}}{p}[\ptpot_{\bm{s}}]](\bm{t})
    =
    \local{\bm{c}}{q}\left[
      \sum_{n=p+1}^\infty \sqrt{\frac{4\pi}{2n+1}} \norm{\bm{s}}^{n} I_n^0(\cdot)
      \right](\bm{t}).
  \]
  Since the series for the right-hand side can be shown to converge uniformly on $\mathbb{R}^3 \setminus B(\bm{0}, r')$ for all $r' > r$ (cf.~Remark~\ref{rem:multipole-analogue}), in particular on $\bar{B}(\bm{c}, \norm{\bm{t}-\bm{c}})$, we may interchange the order of expansion and summation (cf.~Lemma~\ref{lem:locexp-uniform-limit}), obtaining
  \[
    \local{\bm{c}}{q}[\ptpot_{\bm{s}}](\bm{t})
    -
    \local{\bm{c}}{q}[\mpole{\bm{0}}{p}[\ptpot_{\bm{s}}]](\bm{t})
    =
      \sum_{n=p+1}^\infty \sqrt{\frac{4\pi}{2n+1}} \norm{\bm{s}}^{n}
    \local{\bm{c}}{q}[I_n^0](\bm{t}).
  \]
  Using Lemma~\ref{lem:in0-local},
  \[
      \norm{\local{\bm{c}}{q}[I_n^0](\bm{t})}
      \leq \sqrt{\frac{2n+1}{4\pi}} \left(\norm{\bm{c}} - \norm{\bm{t} - \bm{c}}\right)^{-(n+1)}
      \leq
      \sqrt{\frac{2n+1}{4\pi}} R^{-(n+1)}.
  \]
Therefore,
\begin{align*}
    \norm{
    \local{\bm{c}}{q}[\ptpot_{\bm{s}}](\bm{t})
    -
    \local{\bm{c}}{q}[\mpole{\bm{0}}{p}[\ptpot_{\bm{s}}]](\bm{t})
    }
    &= \norm{\sum_{n=p+1}^\infty \sqrt{\frac{4\pi}{2n+1}} \norm{\bm{s}}^n
    \local{\bm{c}}{q}[I_n^0](\bm{t})
    } \nonumber \\
    &\leq
    \sum_{n=p+1}^\infty
    \sqrt{\frac{4\pi}{2n+1}} \norm{\bm{s}}^n
    \norm{\local{\bm{c}}{q}[I_n^0](\bm{t})} \nonumber \\
    &\leq \sum_{n=p+1}^\infty \norm{\bm{s}}^n R^{-(n+1)} \nonumber \\
    &\leq \frac{1}{R-r} \left(\frac{r}{R}\right)^{p+1}.
\end{align*}
\end{proof}

\begin{theorem}[{Source $\to$ Multipole($p$) $\to$ Local($p$) $\to$ Local($q$), cf.~\cite[Hyp.~3]{wala:2019:gigaqbx3d}}]%
\label{thm:m2l2qbxl}%
Let $R > r > 0$ and $R' > r' > 0$.
Let $p, q \geq 0$ be integers.
Let $\bm{s}, \bm{c}, \bm{c'}, \bm{t} \in \mathbb{R}^3$.
Furthermore assume $\norm{\bm{s}} \leq r$, $\norm{\bm{c}} = R' + r = R + r'$, and $\bm{t} \in \bar{B}(\bm{c'}, r' - \norm{\bm{c} - \bm{c'}}) \subseteq \bar{B}(\bm{c}, r')$. Consider the following translation sequence depicted in Figure~\ref{fig:m2l2qbxl}.
First,  a $p$-th order multipole expansion
is formed at the origin of the potential due to~$\bm{s}$.
Second, this is translated to a $p$-th order local expansion centered at $\bm{c}$.
Last, this is translated to a $q$-th order local expansion centered at~$\bm{c'}$, yielding $\local{\bm{c'}}{q}[\local{\bm{c}}{p}[\mpole{\bm{0}}{p}[\phi_{\bm{s}}]]]$.
An error bound for approximating the $q$-th order local expansion $\local{\bm{c'}}{q}[\phi_{\bm{s}}](\bm{t})$ using
$\local{\bm{c'}}{q}[\local{\bm{c}}{p}[\mpole{\bm{0}}{p}[\phi_{\bm{s}}]]](\bm{t})$ is as follows:
\begin{equation}
\label{eqn:m2l-bound}
\norm{\local{\bm{c'}}{q}[\phi_{\bm{s}}](\bm{t}) - \local{\bm{c'}}{q}[\local{\bm{c}}{p}[\mpole{\bm{0}}{p}[\phi_{\bm{s}}]]](\bm{t})} \leq \frac{1}{R - r} \left(\frac{r}{R}\right)^{p+1} + \frac{1}{R' - r'}\left(\frac{r'}{R'}\right)^{p+1}.
\end{equation}
\end{theorem}

\begin{proof}

By adding and subtracting $\local{\bm{c'}}{q}[\mpole{\bm{0}}{p}[\phi_{\bm{s}}]](\bm{t})$ to the left-hand side of~\eqref{eqn:m2l-bound}, we obtain
\begin{multline}
\label{eqn:m2l-triangle-bound}
\norm{\local{\bm{c'}}{q}[\phi_{\bm{s}}](\bm{t}) - \local{\bm{c'}}{q}[\local{\bm{c}}{p}[\mpole{\bm{0}}{p}[\phi_{\bm{s}}]]](\bm{t})}
\leq \norm{\local{\bm{c'}}{q}[\phi_{\bm{s}}](\bm{t})
-
\local{\bm{c'}}{q}[\mpole{\bm{0}}{p}[\phi_{\bm{s}}]](\bm{t})}\\
+
\norm{
\local{\bm{c'}}{q}[\mpole{\bm{0}}{p}[\phi_{\bm{s}}]](\bm{t})
-
\local{\bm{c'}}{q}[\local{\bm{c}}{p}[\mpole{\bm{0}}{p}[\phi_{\bm{s}}]]](\bm{t})
}.
\end{multline}

By Theorem~\ref{thm:m2qbxl},
\begin{equation}
\label{eqn:m2l-triangle-bound-rhs-part1}
\norm{\local{\bm{c'}}{q}[\phi_{\bm{s}}](\bm{t})
-
\local{\bm{c'}}{q}[\mpole{\bm{0}}{p}[\phi_{\bm{s}}]](\bm{t})}
\leq \frac{1}{R-r} \left(\frac{r}{R}\right)^{p+1}.
\end{equation}
In the remainder of this proof, we  show
\begin{equation}
\label{eqn:m2l-triangle-bound-rhs-part2}
\norm{
\local{\bm{c'}}{q}[\mpole{\bm{0}}{p}[\phi_{\bm{s}}]](\bm{t})
-
\local{\bm{c'}}{q}[\local{\bm{c}}{p}[\mpole{\bm{0}}{p}[\phi_{\bm{s}}]]](\bm{t})
} \leq \frac{1}{R'-r'} \left(\frac{r'}{R'}\right)^{p+1},
\end{equation}
from which~\eqref{eqn:m2l-bound} follows via~\eqref{eqn:m2l-triangle-bound}.

We now state a number of facts concerning $\mpole{\bm{0}}{p}[\phi_{\bm{s}}]$. Without loss of generality, in the remainder of the proof we assume $\bm{\hat{s}} = \transpose{(0, 0, 1)}$. Recall that
(cf.~\eqref{eqn:m2qbxl-multipole} in Theorem~\ref{thm:m2qbxl})
\begin{equation}
\label{eqn:mpole-in0}
   \mpole{\bm{0}}{p}[\phi_{\bm{s}}](\bm{x})
   =
   \sum_{n=0}^p \sqrt{\frac{4\pi}{2n+1}} \norm{\bm{s}}^n I_n^0(\bm{x}).
\end{equation}
Recall also the local expansion of $I_n^0$ from Lemma~\ref{lem:in0-local},
\[
\local{\bm{c}}{p}[I_n^0](\bm{t})
= \sum_{\nu=0}^{p} S^I_{n,\nu}(\bm{c}, \bm{t}),
\]
where the terms $S^I_{n,\nu}(\bm{c}, \bm{t})$ are
\[
S^I_{n,\nu}(\bm{c}, \bm{t})
=
A_n^0
\sum_{\mu \in I(0,\nu)}
(-1)^{\nu-\mu}
\binom{n+\nu-\mu}{n}
\frac{I^\mu_{n+\nu}(\bm{c}) R_{\nu}^{-\mu}(\bm{t} -\bm{c})}{A^{\mu}_{n+\nu} A^{-\mu}_\nu},
\quad \nu \in \mathbb{N}_0.
\]

Using~\eqref{eqn:mpole-in0} and via the linearity
of $\local{\bm{c'}}{q}[\cdot]$ and $\local{\bm{c}}{p}[\cdot]$,
\begin{equation}
\label{eqn:m2l-triangle-bound-rhs-part2-expanded}
\local{\bm{c'}}{q}[\mpole{\bm{0}}{p}[\phi_{\bm{s}}]](\bm{t})
-
\local{\bm{c'}}{q}[\local{\bm{c}}{p}[\mpole{\bm{0}}{p}[\phi_{\bm{s}}]]](\bm{t})
=
\sum_{n=0}^p \sqrt{\frac{4\pi}{2n+1}} \norm{\bm{s}}^n
\local{\bm{c'}}{q}[I_n^0 - \local{\bm{c}}{p}[I_n^0]](\bm{t}).
\end{equation}
Observe
\begin{equation}
\label{eqn:in0-expansion}
I_n^0(\bm{t}) - \local{\bm{c}}{p}[I_n^0](\bm{t})
=
\sum_{\nu=p+1}^\infty S^I_{n,\nu}(\bm{c},\bm{t}).
\end{equation}
Applying $\local{\bm{c'}}{q}$ to the right-hand side of~\eqref{eqn:in0-expansion} and interchanging summation and expansion, which is possible due to uniform convergence of the summation, results in
\begin{equation}
\label{eqn:local-expansion-of-in0-error}
\local{\bm{c'}}{q}[I_n^0 - \local{\bm{c}}{p}[I_n^0]](\bm{t})
=
\sum_{\nu=p+1}^\infty
\local{\bm{c'}}{q}\left[S^I_{n,\nu}(\bm{c}, \cdot) \right](\bm{t}).
\end{equation}

If $\bm{c} \neq \bm{c'}$, let $Z$ be a rotation matrix that rotates the vector $\bm{c'} - \bm{c}$ parallel to the $z$-axis, i.e., such that $Z(\widehat{\bm{c'}-\bm{c}}) = \transpose{(0, 0, 1)}$. If $\bm{c} = \bm{c'}$, for definiteness let $Z$ be the identity matrix. For $\nu \in \mathbb{N}_0$, it can be shown that the rotated spherical harmonics $\{ Y_\nu^\mu(Z \cdot) : \mu \in \mathbb{N}_0, -\nu \leq \mu \leq \nu \}$ form an orthonormal basis of $\mathbb{Y}_{\nu}^3$ (\cite[Ch.~2]{atkinson:2012:sph-harm}). We define a $(2\nu + 1) \times (2\nu + 1)$ matrix $C_{\nu} = \{c_\nu^{\mu,\lambda}\}$ relating the basis sets $\{ Y_\nu^\mu(Z \cdot) \}$ and $\{ Y_{\nu}^{\mu}\}$. The rows and columns of $C_{\nu}$ are indexed by $-\nu, -\nu + 1, \ldots, \nu$. Note the negative starting index. The entry at row $\mu$, column $\lambda$ satisfies
\[
    c_{\nu}^{\mu, \lambda} = \innerprod{Y_{\nu}^{\mu}}{Y_{\nu}^{\lambda}(Z\cdot)}
\]
so that, for all $\bm{\xi} \in \mathbb{S}^2$,
\begin{equation}
  \label{eqn:rotated-sph-harm-expansion}
   Y_{\nu}^{\mu}(\bm{\xi})
   = \sum_{\lambda=-\nu}^{\nu} c_{\nu}^{\mu,\lambda} Y_{\nu}^{\lambda}(Z \bm{\xi}).
\end{equation}
Then $C_{\nu}$ is a unitary matrix (cf.~\cite[eq.~(2.26)]{atkinson:2012:sph-harm}), i.e., $\htranspose{C_{\nu}}C_{\nu} = I$, where $\htranspose{(\cdot)}$ is the conjugate transpose.

In the next step, we expand $\local{\bm{c'}}{q}[S^I_{n,\nu}(\bm{c}, \cdot)]$ in the rotated basis. Using the addition theorem for solid harmonics~\eqref{eqn:rnm-addition-thm} and~\eqref{eqn:rotated-sph-harm-expansion},
\begin{align}
\local{\bm{c'}}{q}[R_{\nu}^{-\mu}(\cdot - \bm{c})](\bm{t})
&=
\local{\bm{c'}}{q}
\left[
\sum_{\lambda=-\nu}^{\nu}
 c_{\nu}^{-\mu,\lambda} R_{\nu}^{\lambda}(Z(\cdot - \bm{c}))
\right](\bm{t})\nonumber \\
& =
\label{eqn:subexpansion-unsimplified}
\sum_{\nu'=0}^{\min(\nu,q)}
\sum_{\lambda=-\nu}^{\nu}
A_{\nu}^{\lambda}
c_{\nu}^{-\mu,\lambda}
\\&\quad \quad \sum_{\ilocloc{\lambda'}{\nu'}{\lambda}{\nu-\nu'}}
\binom{\nu+\lambda}{\nu'+\lambda'}
\frac{
R^{\lambda-\lambda'}_{\nu-\nu'}(Z(\bm{c'}-\bm{c}))
R^{\lambda'}_{\nu'}(Z(\bm{t} - \bm{c'}))
}{
A^{\lambda-\lambda'}_{\nu-\nu'}
A^{\lambda'}_{\nu'}
} \nonumber\\
&=
\label{eqn:subexpansion-simplified}
\sum_{\nu'=0}^{\min(\nu,q)}
\sum_{\lambda=-\nu}^{\nu}
A_{\nu}^{\lambda}
c_{\nu}^{-\mu,\lambda}
\norm{\bm{c'}-\bm{c}}^{\nu-\nu'}
\sum_{\mathclap{\substack{\lambda' \in I(0, \nu') \\
\cap \{\lambda\}}}}
\hspace{0.5em}
\binom{\nu+\lambda}{\nu'+\lambda'}
\frac{
R^{\lambda'}_{\nu'}(Z(\bm{t} - \bm{c'}))
}{
A^{\lambda'}_{\nu'}
}\\
&=
\sum_{\nu'=0}^{\min(\nu, q)}
\sum_{\lambda=-\nu'}^{\nu'}
A^{\lambda}_{\nu}
c_{\nu}^{-\mu,\lambda}
\norm{\bm{c'}-\bm{c}}^{\nu-\nu'}
\binom{\nu+\lambda}{\nu'+\lambda}
\frac{R^{\lambda}_{\nu'}(Z(\bm{t}-\bm{c'}))}{
A^{\lambda}_{\nu'}}\nonumber.
\end{align}
To go from~\eqref{eqn:subexpansion-unsimplified} to~\eqref{eqn:subexpansion-simplified}, notice the inner summation is zero unless $\lambda'=\lambda$ [cf.~\eqref{eqn:sph-harm-z-axis}].)
(If $\bm{c}=\bm{c'}$, additionally it is zero unless $\nu=\nu'$.) It follows that we can write $\local{\bm{c'}}{q}[S^I_{n,\nu}(\bm{c}, \cdot)](\bm{t})$ as
\begin{align*}
\local{\bm{c'}}{q}[S^I_{n,\nu}(\bm{c}, \cdot)](\bm{t}) &=
A_n^0 \sum_{\nu'=0}^{\min(\nu, q)}
\sum_{\mu=-\nu}^{\nu}
(-1)^{\nu-\mu}
\binom{n+\nu-\mu}{n}
\frac{I^{\mu}_{n+\nu}(\bm{c})}{
A^{\mu}_{n+\nu} A^{-\mu}_{\nu}} \\
&\quad \quad
\sum_{\lambda=-\nu'}^{\nu'}
A^{\lambda}_{\nu}
c_{\nu}^{-\mu,\lambda}
\norm{\bm{c'}-\bm{c}}^{\nu-\nu'}
\binom{\nu+\lambda}{\nu'+\lambda}
\frac{R^{\lambda}_{\nu'}(Z(\bm{t}-\bm{c'}))}{
A^{\lambda}_{\nu'}}.
\end{align*}

Expanding, rearranging, and simplifying the previous equation,
we obtain
\begin{equation}
\label{eqn:snnu-expansion-simplified}
\local{\bm{c'}}{q}[S^I_{n,\nu}(\bm{c}, \cdot)](\bm{t})
= A_n^0 \sum_{\nu'=0}^{\min(\nu, q)}
\sum_{\mu=-\nu}^{\nu}
(-1)^{\nu-\mu}
\alpha^{\mu}_{\nu,\nu'}(\bm{c})
\beta_{\nu,\nu'}^{\mu}(\bm{c'}-\bm{c},Z(\bm{t}-\bm{c'}))
\end{equation}
where
\begin{align*}
    \alpha_{\nu,\nu'}^{\mu}(\bm{x}) &=
    \sqrt{\frac{4\pi}{2(n+\nu)+1}}
    \sqrt{\binom{n+\nu-\mu}{n} \binom{n+\nu+\mu}{n}}
    \norm{\bm{x}}^{-(n+\nu+1)}
    Y_{n+\nu}^{\mu}(\hat{\bm{x}}),\\
    \beta_{\nu,\nu'}^{\mu}(\bm{x},\bm{y}) &=
    \sqrt{\frac{4\pi}{2\nu'+1}}
    \norm{\bm{x}}^{\nu-\nu'}
    \norm{\bm{y}}^{\nu'}
    \sum_{\lambda=-\nu'}^{\nu'}
    c_\nu^{-\mu,\lambda}
    \sqrt{\binom{\nu+\lambda}{\nu-\nu'}
    \binom{\nu-\lambda}{\nu-\nu'}}
    Y_{\nu'}^{\lambda}(\hat{\bm{y}}).
\end{align*}
Observe that
\begin{align}
    \label{eqn:alphamnnp-sum-sq}
    \sum_{\mu=-\nu}^{\nu} \norm{\alpha_{\nu,\nu'}^\mu(\bm{x})}^2
    &\leq
    \left(
    \sqrt{\frac{4\pi}{2(n+\nu)+1}} \binom{n+\nu}{n}
    \norm{\bm{x}}^{-(n+\nu+1)}
    \right)^2
    \sum_{\mu=-\nu}^{\nu}
    \norm{Y_{n+\nu}^\mu(\hat{\bm{x}})}^2\\
    &= \left[\binom{n+\nu}{n}
    \norm{\bm{x}}^{-(n+\nu+1)}\right]^2. \nonumber
\end{align}
Also,
\begin{align}
    \label{eqn:betamnnp-sum-sq}
    \sum_{\mu=-\nu}^{\nu} \norm{\beta_{\nu,\nu'}^\mu(\bm{x},\bm{y})}^2
    &=
    \left(\sqrt{\frac{4\pi}{2\nu'+1}}
    \norm{\bm{x}}^{\nu-\nu'}
    \norm{\bm{y}}^{\nu'}\right)^2 \\
    &\quad\quad
    \sum_{\mu=-\nu}^{\nu}
    \left|
    \sum_{\lambda=-\nu'}^{\nu'}
    c_\nu^{-\mu,\lambda}
    \sqrt{\binom{\nu+\lambda}{\nu-\nu'}
    \binom{\nu-\lambda}{\nu-\nu'}}
    Y_{\nu'}^{\lambda}(\hat{\bm{y}})
    \right|^2.\nonumber
\end{align}
Define a vector-valued function $\bm{v}_{\nu,\nu'}: \mathbb{S}^2 \to \mathbb{C}^{2\nu+1}$, with output vector indexed by $\{-\nu,-\nu+1, \ldots \nu\}$,
whose $\lambda$-th entry is
\[
\left[v_{\nu,\nu'}(\bm{\xi})\right]_{\lambda} =
\begin{dcases}
\sqrt{\binom{\nu+\lambda}{\nu-\nu'}
\binom{\nu-\lambda}{\nu-\nu'}}
Y_{\nu'}^\lambda(\bm{\xi}), & \lambda \in I(0,\nu'),\\
0, & \lambda \in I(0, \nu) \setminus I(0, \nu').
\end{dcases}
\]
In~\eqref{eqn:betamnnp-sum-sq} the summation on the right-hand side of the equation satisfies
\[
    \sum_{\mu=-\nu}^{\nu}
    \left|
    \sum_{\lambda=-\nu'}^{\nu'}
    c_\nu^{-\mu,\lambda}
    \sqrt{\binom{\nu+\lambda}{\nu-\nu'}
    \binom{\nu-\lambda}{\nu-\nu'}}
    Y_{\nu'}^{\lambda}(\hat{\bm{y}})
    \right|^2
    =
    \htranspose{(C_{\nu}\bm{v}_{\nu,\nu'}(\bm{\hat{y}}))}(C_{\nu}\bm{v}_{\nu,\nu'}(\bm{\hat{y}})).
\]
Since $C_{\nu}$ is unitary,
\[
    \htranspose{(C_{\nu}\bm{v}_{\nu,\nu'}(\bm{\hat{y}}))}(C_{\nu}\bm{v}_{\nu,\nu'}(\bm{\hat{y}}))
    =
   \htranspose{\bm{v}_{\nu,\nu'}(\bm{\hat{y}})}
   \bm{v}_{\nu,\nu'}(\bm{\hat{y}})
    =
   \norm{\bm{v}_{\nu,\nu'}(\bm{\hat{y}})}^2
\]
We also have
\begin{align}
\label{eqn:v-norm-bound}
\norm{\bm{v}_{\nu,\nu'}(\bm{\hat{y}})}^2
&= \sum_{\lambda=-\nu'}^{\nu'}
\binom{\nu+\lambda}{\nu-\nu'}
\binom{\nu-\lambda}{\nu-\nu'}
\norm{Y_{\nu'}^{\lambda}(\bm{\hat{y}})}^2
\\
&\leq \binom{\nu}{\nu'}^2 \frac{2\nu'+1}{4\pi}.\nonumber
\end{align}
Combining~\eqref{eqn:v-norm-bound} and~\eqref{eqn:betamnnp-sum-sq},
\begin{equation}
\label{eqn:betamnnp-sum-sq-part2}
    \sum_{\mu=-\nu}^{\nu} \norm{\beta_{\nu,\nu'}^\mu(\bm{x},\bm{y})}^2
    \leq
    \left[\binom{\nu}{\nu'} \norm{\bm{x}}^{\nu-\nu'} \norm{\bm{y}}^{\nu'}\right]^2.
\end{equation}

Applying the Cauchy--Schwarz inequality to~\eqref{eqn:snnu-expansion-simplified}, along with~\eqref{eqn:alphamnnp-sum-sq}~and~\eqref{eqn:betamnnp-sum-sq-part2},
\begin{align}
\label{eqn:local-expansion-of-snnu-bound}
\norm{\local{\bm{c'}}{q}[S^I_{n,\nu}(\bm{c}, \cdot)](\bm{t})}
&\leq
A_n^0 \hspace{-0.7em} \sum_{\nu'=0}^{\min(\nu,q)}
\left|
\sum_{\mu=-\nu}^{\nu}
(-1)^{\nu-\mu}
\alpha_{\nu,\nu'}^{\mu}(\bm{c})
\beta_{\nu,\nu'}^{\mu}(\bm{c'}-\bm{c},Z(\bm{t}-\bm{c'}))
\right| \\
&\leq
A_n^0 \hspace{-0.7em} \sum_{\nu'=0}^{\min(\nu,q)}
\hspace{-0.5em}
\sqrt{
\left(
\sum_{\mu=-\nu}^{\nu}
\norm{\alpha_{\nu,\nu'}^{\mu}(\bm{c})}^2
\right)
\left(
\sum_{\mu=-\nu}^{\nu}
\norm{\beta_{\nu,\nu'}^{\mu}(\bm{c'}-\bm{c},Z(\bm{t}-\bm{c'}))}^2
\right)} \nonumber\\
&= A_n^0 \binom{n+\nu}{n} \norm{\bm{c}}^{-(n+\nu+1)}
\sum_{\nu'=0}^{\min(\nu,q)}
\binom{\nu}{\nu'} \norm{\bm{c'}-\bm{c}}^{\nu - \nu'} \norm{\bm{t}-\bm{c'}}^{\nu'}
\nonumber \\
&\leq \sqrt{\frac{2n+1}{4\pi}}\binom{n+\nu}{n} \norm{\bm{c}}^{-(n+\nu+1)} \left( \norm{\bm{c'}-\bm{c}} + \norm{\bm{t}-\bm{c'}}\right)^{\nu}. \nonumber
\end{align}
We can now bound~\eqref{eqn:m2l-triangle-bound-rhs-part2-expanded} as follows:
\begin{align}
\norm{
    \local{\bm{c'}}{q}[
    \mpole{\bm{0}}{p}[\phi_{\bm{s}}]](\bm{t})-
    \local{\bm{c'}}{q}[\local{\bm{c}}{p}[\mpole{\bm{0}}{p}[\phi_{\bm{s}}]]](\bm{t})
    }
    &=
\norm{
\sum_{n=0}^p
\sqrt{\frac{4\pi}{2n+1}}
\norm{\bm{s}}^n
\local{\bm{c'}}{q}[I_n^0 - \local{\bm{c}}{p}[I_n^0]](\bm{t})
} \nonumber \\
&\leq
\sum_{n=0}^\infty
\sqrt{\frac{4\pi}{2n+1}}
\norm{\bm{s}}^n
\norm{
\local{\bm{c'}}{q}[I_n^0 - \local{\bm{c}}{p}[I_n^0]](\bm{t})
} \nonumber \\
&=
\sum_{n=0}^\infty
\sqrt{\frac{4\pi}{2n+1}}
\norm{\bm{s}}^n
\norm{\sum_{\nu=p+1}^\infty \local{\bm{c'}}{q} \left[ S_{n,\nu}(\bm{c}, \cdot) \right]
  (\bm{t})} \label{eqn:tricky-bound-part1} \\
&\leq
\sum_{n=0}^\infty
\norm{\bm{s}}^n
\sum_{\nu=p+1}^\infty \binom{n+\nu}{n} \norm{\bm{c}}^{-(n+\nu+1)} (r')^{\nu}
\label{eqn:tricky-bound-part2} \\
&=
\sum_{n=0}^\infty
\frac{\norm{\bm{s}}^n}{\norm{\bm{c}}^n}
\sum_{\nu=p+1}^\infty \binom{n+\nu}{n} \frac{(r')^{\nu}}{\norm{\bm{c}}^{\nu+1}}, \nonumber
\end{align}
where~\eqref{eqn:tricky-bound-part1} follows from~\eqref{eqn:local-expansion-of-in0-error}
and~\eqref{eqn:tricky-bound-part2} follows from~\eqref{eqn:local-expansion-of-snnu-bound}
and the fact that $\norm{\bm{c'}-\bm{c}} + \norm{\bm{t} - \bm{c'}} \leq r'$.
After interchanging the  summations in the last line,
\begin{align*}
\norm{
    \local{\bm{c'}}{q}[
    \mpole{\bm{0}}{p}[\phi_{\bm{s}}]](\bm{t})-
    \local{\bm{c'}}{q}[\local{\bm{c}}{p}[\mpole{\bm{0}}{p}[\phi_{\bm{s}}]]](\bm{t})
    }
    &\leq
\sum_{\nu=p+1}^\infty
\frac{(r')^{\nu}}{\norm{\bm{c}}^{\nu+1}}
\sum_{n=0}^\infty
\binom{n+\nu}{n}
\frac{\norm{\bm{s}}^n}{\norm{\bm{c}}^n}
  \nonumber \\
&=
\sum_{\nu=p+1}^\infty
\frac{(r')^{\nu}}{\norm{\bm{c}}^{\nu+1}}
\frac{1}{\left(1 - \frac{\norm{\bm{s}}}{\norm{\bm{c}}}\right)^{\nu+1}}\nonumber \\
&\leq
\frac{1}{R'} \sum_{\nu=p+1}^\infty \frac{(r')^{\nu}}{(R')^{\nu}} \nonumber \\
&= \frac{1}{R' - r'} \left(\frac{r'}{R'}\right)^{p+1}. \nonumber
\end{align*}
\end{proof}

\subsubsection{Remarks}

The theorems stated above may be seen as generalizations of the Greengard--Rokhlin error estimates for approximation of point potentials. The reason for this is that a local expansion $\local{\bm{c}}{q}[f]$ evaluated at its center $\bm{c}$ is the point potential $f(\bm{c})$, i.e., the point potential may be regarded as a local expansion of radius zero. The error estimates for evaluation of truncated local and multipole expansions, in Theorems~\ref{thm:l2qbxl} and~\ref{thm:m2qbxl}, imply the \emph{same} error bound as the analogous truncation estimates for point potentials in~\cite[Lem.~3.4.2]{greengard:1988:thesis}.

In the multipole-to-local case, there are some differences between Theorem~\ref{thm:m2l2qbxl} and the multipole-to-local error estimate in~\cite[Thm.~3.5.5]{greengard:1988:thesis} (cf.~\cite[Thm.~5.4]{beatson:greengard:1997}). An obvious but inessential difference is that the latter only considers multipole and local expansions balls of the same radius, while Theorem~\ref{thm:m2l2qbxl} lets their radii vary. Restated in the language of this paper, the latter only considers the case $R'=R$ and $r'=r$.

For the case that the local and multipole expansion balls are the same radius $r$ and have a separation distance of $R - r$,~\cite[Thm.~3.5.5]{greengard:1988:thesis} implies a truncation error bound for a unit-strength point source of $1/(R-r) (r/R)^{p+1}$. This is \emph{half} as large as the bound that~\eqref{eqn:m2l-bound} implies for the same evaluation scenario.\footnote{The bound $1/(R-r) (r/R)^{p+1}$ can be violated when a multipole center, source, target, and (intermediate) local expansion center are arranged in order in a line. Numerical evidence suggests $2/(R-r) (r/R)^{p+1}$ is sharp, and it should be possible to show based on arguments in this paper.} The reason for this is that~\cite[Thm.~3.5.5]{greengard:1988:thesis} appears only to  model the impact of one intermediate truncation, while for the multipole-to-local case there are two relevant intermediate truncations which introduce error. Our bound takes both into account.

We have already briefly mentioned the two-dimensional case in Remark~\ref{rem:generalization}. In that case, the Fourier modes are the analogues to the spherical harmonics. Two-dimensional Laplace potentials can be expanded in these modes, typically in complex variables~\cite{carrier:1988:adaptive-fmm}. There is a two-dimensional analogue to the theorems stated above. If correctly carried out using the analogous arguments, the 2D bounds for GIGAQBX stated in~\cite{wala:2018:gigaqbx2d} can be improved to eliminate excess leading factors of $p$ and $q$ found in that paper. We leave the details to the reader.

\subsection{Applications to GIGAQBX}

As previously mentioned, the results of the previous section prove Hypotheses 1--3 in~\cite{wala:2019:gigaqbx3d}. A consequence of these is~\cite[Thm.~1]{wala:2019:gigaqbx3d}, which gives an accuracy bound on GIGAQBX~FMM and was stated conditionally based on their truth. Given the truth of these hypotheses, we can restate the theorem unconditionally.

For context, we recall the basic computational setup. To evaluate the single-layer potential $\mathcal{S} \sigma$, GIGAQBX takes as input a discretization of the input geometry $\Gamma$ into $N_S$ source quadrature nodes with corresponding quadrature weights $\{w_i\}_{i=1}^{N_S}$, and nodal density values of the source density. It also takes a set of evaluation targets and QBX expansion centers and radii, along with a mapping from targets to QBX centers. An FMM order $p > 0$ and a QBX order $q > 0$ are also specified. Lastly, a \emph{target confinement factor} $t_f > 0$ is specified.

The target confinement factor is a parameter unique to GIGAQBX, and deserves some explanation, for it is the primary means by which accuracy guarantees are enforced. Recall that, like other FMMs, GIGAQBX uses a quadtree/octree as its computational domain, constructed by recursively partitioning square boxes into equal-sized children until the number of particles per box is below a predetermined cutoff. The target confinement factor establishes a \emph{target confinement region} for a box $b$ of radius $\norm{b}$, which, in three dimensions, is a ball of radius $\sqrt{3}\norm{b}(1+t_f)$ centered at the box center. During octree construction, a QBX expansion ball is not allowed to exceed its box's target confinement region. If placing it in the box would exceed the TCR, it remains placed in the parent box.

\begin{theorem}[{Accuracy of GIGAQBX, cf.~\cite[Thm.~1]{wala:2019:gigaqbx3d}}]%
Fix a target confinement factor $0 \leq t_f < 2 \sqrt{3} - 2 \approx 1.47$. Define
\[
A = \bignorm{\sigma}_\infty \sum_{i=1}^{N_S} \norm{w_i},
\]
and let $R$ be the minimum box radius in the tree. Let $\Phi$ denote the point potential approximation to the single-layer potential $\mathcal{S} \sigma$ (cf.~\eqref{eqn:slp-quad}), and let $\mathcal{A}_{\mathrm{GIGA}}^p[\Phi]$ denote the approximation to this point potential as computed via GIGAQBX\@. Then a constant $M >0$ exists, independent of $t_f$, $\sigma$, the particle distribution, and the QBX/FMM orders, such that the absolute acceleration error in the GIGAQBX FMM at a target $\bm{t}$ associated with a center $\bm{c}$ is bounded by
\begin{multline*}
\norm{\local{\bm{c}}{q}[\Phi](\bm{t})
-
\local{\bm{c}}{q}[\mathcal{A}^p_{\mathrm{GIGA}}[\Phi]](\bm{t})}\\
\leq
\frac{AM}{R} \max \left(
\frac{1}{3 - \sqrt{3}} \left(\frac{\sqrt{3}}{3}\right)^{p+1},
\frac{1}{6-2\sqrt{3}-\sqrt{3}t_f} \left(\frac{\sqrt{3}(1+t_f)}{6 - \sqrt{3}} \right)^{p+1}
\right).
\end{multline*}
\end{theorem}

\begin{proof}
See~\cite{wala:2019:gigaqbx3d}. The only change in notation is the use of the expression $\local{\bm{c}}{q}[\Phi](\bm{t})
-
\local{\bm{c}}{q}[\mathcal{A}^p_{\mathrm{GIGA}}[\Phi]](\bm{t})$ to represent the acceleration error. We have also rearranged the quantities on the right-hand side algebraically to yield a slightly improved final estimate, with no change in the proof itself.
\end{proof}

\section{Numerical validation}
\label{sec:experiments}

The paper~\cite{wala:2019:gigaqbx3d} reports the results of a numerical
study of error in the approximation of local expansions of the 3D~Laplace kernel
using Greengard--Rokhlin translation operators. The setting for the numerical study almost directly mirrors
the settings for Theorems~\ref{thm:m2qbxl},~\ref{thm:l2qbxl},
and~\ref{thm:m2l2qbxl} of
Section~\ref{sec:approximation-results}. Table~\ref{tab:theoretical-vs-empirical-ests} states the numerically derived bounds from that paper. The numerical experiments to obtain these results can be found at~\url{https://github.com/mattwala/gigaqbx-accuracy-experiments}.

The experiment is based on a sampling procedure that samples the error using a number of
geometrical source/target/center positions for each evaluation scenario. The `source' order $p$ and the `target' order $q$ are chosen from the set $\{3, 5, 10, 15, 20\}$. The output of the computation is an estimate of the constant in front of the expression for the error. In the previous section, it was shown that this constant must be $1$ in all cases. Numerically, the value is slightly above $1$, but nevertheless in close agreement with the theoretical value. The reason for the overshoot in the constant appears to be due to the effects of finite precision, as in a number of evaluation scenarios, the error is very small in comparison with the computed value of the potential, whereby a combination of floating point rounding and cancellation could lead to the result.

\begin{table}
  \centering
  \caption{Numerically derived estimates
    in~\cite{wala:2019:gigaqbx3d} for acceleration errors for in Greengard--Rokhlin translation operators for the 3D Laplace kernel. The leading
    constant was numerically derived using a sampling procedure. The notation has been updated to be consistent with the notation in this paper. A leading factor of $1/(4\pi)$ present in the original bounds is removed.
    }
  \label{tab:theoretical-vs-empirical-ests}
  \begin{tabular}{lll}
    \toprule
    Interaction & Numerical error bound from~\cite{wala:2019:gigaqbx3d} \\
    \midrule
    M($p$)$\to$L($q$) (Thm.~\ref{thm:m2qbxl})
    &
    $\displaystyle \frac{1.002}{R - r} \left(\frac{r}{R}\right)^{p+1}$ \\
    L($p$)$\to$L($q$) (Thm.~\ref{thm:l2qbxl})
    &
    $\displaystyle \frac{1.001}{R - r} \left(\frac{r}{R}\right)^{p+1}$ \\
    M($p$)$\to$L($p$)$\to$L($q$) (Thm.~\ref{thm:m2l2qbxl})
    &$\displaystyle \frac{1.001}{R - r} \left( \frac{r}{R} \right)^{p+1} + \frac{1.001}{R' - r'} \left( \frac{r'}{R'} \right)^{p+1}$
    \\
    \bottomrule
  \end{tabular}
\end{table}

\section{Conclusions}
\label{sec:conclusions}

In this paper we have analytically examined what happens when the FMM is modified to output
local expansions rather than point potentials. Our main result is that, in the Laplace case, local expansions under suitably reinterpreted evaluation scenarios exhibit an error bound that very strongly resembles that of point potentials. At a high level, this confirms the intuition---made explicit in the design of GIGAQBX---that they may be treated as a `target with extent.'

This work raises some interesting questions.

\emph{Can the 2D/3D cases be unified?} The techniques in this paper can be used to establish basically the same results for Laplace potentials in two dimensions---incidentally, these improve the original 2D results given in~\cite{wala:2018:gigaqbx2d}. The main difference is the use of a different spherical and solid harmonic basis. Nevertheless, a unified high-level treatment seems possible.

\emph{What happens when the expansion is non-convergent?} An assumption our techniques make is that the expansion being approximated converges. This assumption is artificial. It is still possible to define a local expansion in a region where the expansion does not converge. Indeed, in some cases the expansions made by the QBX~FMM~of~\cite{rachh:2017:qbx-fmm} are in regions where the original local expansion does not converge (as $q \to \infty$), as observed in~\cite[Sec.~2.4]{wala:2018:gigaqbx2d}. Understanding this behavior better may lead to more efficient fast algorithms for this case.

\emph{Can these estimates be made sharper?} There are a couple of ways in which these bounds overapproximate. First, they are not based on the sharpest available point FMM estimates~\cite{petersen:1995:fmm-error-est-3d}. Second, they do not make use of the final expansion order $q$ at all, an extra piece of information which could potentially improve the bound.

\emph{What happens for non-Laplace potentials?} The Helmholtz case remains open. A bound similar to that of the one given in Section~\ref{sec:approximation-theory-results} still applies for Helmholtz potentials and other potentials that use spherical harmonic expansions, but it is likely that improvements  can be made, as with our treatment of the Laplace case.
\section*{Acknowledgments}

The authors' research was supported by the National Science Foundation under
awards DMS-1654756 and SHF-1911019 as well as by the Department of Computer Science
at the University of Illinois at Urbana-Champaign. Any opinions, findings, and
conclusions, or recommendations expressed in this article are those of the
authors and do not necessarily reflect the views of the National Science
Foundation; NSF has not approved or endorsed its content. Portions of this work
are based on the first author's Ph.D.\ thesis.


\bibliography{main}

\begin{thebibliography}{10}

\bibitem{anderson:fmm:1992}
{\sc C.~R. Anderson}, {\em An implementation of the fast multipole method
  without multipoles}, SIAM Journal on Scientific and Statistical Computing, 13
  (1992), pp.~923--947, \url{https://doi.org/10.1137/0913055}.

\bibitem{atkinson:2012:sph-harm}
{\sc K.~Atkinson and W.~Han}, {\em Spherical harmonics and approximations on
  the unit sphere: an introduction}, vol.~2044 of Lecture Notes in Mathematics,
  Springer, Heidelberg, 2012, \url{https://doi.org/10.1007/978-3-642-25983-8}.

\bibitem{axler:harmonic-functions:2001}
{\sc S.~Axler, P.~Bourdon, and W.~Ramey}, {\em Harmonic function theory},
  vol.~137 of Graduate Texts in Mathematics, Springer-Verlag, New York,
  second~ed., 2001, \url{https://doi.org/10.1007/978-1-4757-8137-3}.

\bibitem{beatson:greengard:1997}
{\sc R.~Beatson and L.~Greengard}, {\em A short course on fast multipole
  methods}, in Wavelets, multilevel methods and elliptic {PDE}s ({L}eicester,
  1996), Numer. Math. Sci. Comput., Oxford Univ. Press, New York, 1997,
  pp.~1--37.

\bibitem{caola:1978:solid-harmonics}
{\sc M.~J. Caola}, {\em Solid harmonics and their addition theorems}, J. Phys.
  A, 11 (1978), pp.~L23--L25, \url{https://doi.org/10.1088/0305-4470/11/2/001}.

\bibitem{carrier:1988:adaptive-fmm}
{\sc J.~Carrier, L.~Greengard, and V.~Rokhlin}, {\em A fast adaptive multipole
  algorithm for particle simulations}, SIAM J. Sci. Statist. Comput., 9 (1988),
  pp.~669--686, \url{https://doi.org/10.1137/0909044}.

\bibitem{dai:2013:approximation}
{\sc F.~Dai and Y.~Xu}, {\em Approximation theory and harmonic analysis on
  spheres and balls}, Springer Monographs in Mathematics, Springer, New York,
  2013, \url{https://doi.org/10.1007/978-1-4614-6660-4}.

\bibitem{nist:dlmf}
{\em {\it NIST Digital Library of Mathematical Functions}}.
\newblock http://dlmf.nist.gov/, Release 1.0.25 of 2019-12-15,
  \url{http://dlmf.nist.gov/}.
\newblock F.~W.~J. Olver, A.~B. {Olde Daalhuis}, D.~W. Lozier, B.~I. Schneider,
  R.~F. Boisvert, C.~W. Clark, B.~R. Miller, B.~V. Saunders, H.~S. Cohl, and
  M.~A. McClain, eds.

\bibitem{epstein:2013:qbx-error-est}
{\sc C.~L. Epstein, L.~Greengard, and A.~Kl\"{o}ckner}, {\em On the convergence
  of local expansions of layer potentials}, SIAM J. Numer. Anal., 51 (2013),
  pp.~2660--2679, \url{https://doi.org/10.1137/120902859}.

\bibitem{greengard:1988:thesis}
{\sc L.~Greengard}, {\em The rapid evaluation of potential fields in particle
  systems}, ACM Distinguished Dissertations, MIT Press, Cambridge, MA, 1988.

\bibitem{greengard_rokhlin_1997}
{\sc L.~Greengard and V.~Rokhlin}, {\em A new version of the fast multipole
  method for the {Laplace} equation in three dimensions}, Acta Numerica, 6
  (1997), p.~229–269, \url{https://doi.org/10.1017/S0962492900002725}.

\bibitem{gronwall:1914:laplace-series}
{\sc T.~H. Gronwall}, {\em On the degree of convergence of {L}aplace's series},
  Trans. Amer. Math. Soc., 15 (1914), pp.~1--30,
  \url{https://doi.org/10.2307/1988688}.

\bibitem{klockner:2013:qbx}
{\sc A.~Kl\"{o}ckner, A.~Barnett, L.~Greengard, and M.~O'Neil}, {\em Quadrature
  by expansion: a new method for the evaluation of layer potentials}, J.
  Comput. Phys., 252 (2013), pp.~332--349,
  \url{https://doi.org/10.1016/j.jcp.2013.06.027}.

\bibitem{petersen:1995:fmm-error-est-3d}
{\sc H.~G. Petersen, E.~R. Smith, and D.~Soelvason}, {\em Error estimates for
  the fast multipole method. {II}. {The} three-dimensional case}, Proceedings
  of the Royal Society A: Mathematical, Physical and Engineering Sciences, 448
  (1995), pp.~401--418, \url{https://doi.org/10.1098/rspa.1995.0024}.

\bibitem{petersen:1995:fmm-error-est}
{\sc H.~G. Petersen, D.~Soelvason, J.~W. Perram, and E.~R. Smith}, {\em Error
  estimates for the fast multipole method. {I}. {The} two-dimensional case},
  Proceedings of the Royal Society A: Mathematical, Physical and Engineering
  Sciences, 448 (1995), pp.~389--400,
  \url{https://doi.org/10.1098/rspa.1995.0023}.

\bibitem{rachh2015integral}
{\sc M.~Rachh}, {\em Integral equation methods for problems in electrostatics,
  elastostatics and viscous flow}, PhD thesis, New York University, 2015.

\bibitem{rachh:2017:qbx-fmm}
{\sc M.~Rachh, A.~Kl\"{o}ckner, and M.~O'Neil}, {\em Fast algorithms for
  quadrature by expansion {I}: {G}lobally valid expansions}, J. Comput. Phys.,
  345 (2017), pp.~706--731, \url{https://doi.org/10.1016/j.jcp.2017.04.062}.

\bibitem{rahimian:2018:qbkix}
{\sc A.~Rahimian, A.~Barnett, and D.~Zorin}, {\em Ubiquitous evaluation of
  layer potentials using quadrature by kernel-independent expansion}, BIT, 58
  (2018), pp.~423--456, \url{https://doi.org/10.1007/s10543-017-0689-2}.

\bibitem{rivlin:1981:approximation-theory}
{\sc T.~J. Rivlin}, {\em An introduction to the approximation of functions},
  Dover Publications, Inc., New York, 1981.
\newblock Corrected reprint of the 1969 original, Dover Books on Advanced
  Mathematics.

\bibitem{siegel:2018}
{\sc M.~Siegel and A.-K. Tornberg}, {\em A local target specific quadrature by
  expansion method for evaluation of layer potentials in 3{D}}, J. Comput.
  Phys., 364 (2018), pp.~365--392,
  \url{https://doi.org/10.1016/j.jcp.2018.03.006}.

\bibitem{wala:2018:gigaqbx2d}
{\sc M.~Wala and A.~Kl\"{o}ckner}, {\em A fast algorithm with error bounds for
  {Q}uadrature by {E}xpansion}, J. Comput. Phys., 374 (2018), pp.~135--162,
  \url{https://doi.org/10.1016/j.jcp.2018.05.006}.

\bibitem{wala:2019:gigaqbx3d}
{\sc M.~Wala and A.~Kl\"{o}ckner}, {\em A fast algorithm for {Q}uadrature by
  {E}xpansion in three dimensions}, J. Comput. Phys., 388 (2019), pp.~655--689,
  \url{https://doi.org/10.1016/j.jcp.2019.03.024}.

\bibitem{wala:2020:gigaqbxts}
{\sc M.~Wala and A.~Klöckner}, {\em Optimization of fast algorithms for global
  {Quadrature} by {Expansion} using target-specific expansions}, J. Comput.
  Phys., 403 (2020), p.~108976,
  \url{https://doi.org/10.1016/j.jcp.2019.108976}.

\bibitem{white:point-and-shoot:1996}
{\sc C.~A. White and M.~Head‐Gordon}, {\em Rotating around the quartic
  angular momentum barrier in fast multipole method calculations}, The Journal
  of Chemical Physics, 105 (1996), pp.~5061--5067,
  \url{https://doi.org/10.1063/1.472369}.

\bibitem{ying:kifmm:2004}
{\sc L.~Ying, G.~Biros, and D.~Zorin}, {\em A kernel-independent adaptive fast
  multipole algorithm in two and three dimensions}, J. Comput. Phys., 196
  (2004), pp.~591--626, \url{https://doi.org/10.1016/j.jcp.2003.11.021}.

\end{thebibliography}

\end{document}